\documentclass[a4paper,12pt]{article}

\usepackage{amssymb}                                    
\usepackage{mathrsfs}                    
\usepackage{amsmath}                    
\usepackage{amsfonts}                   
\usepackage{amsthm}                     
\usepackage[latin1]{inputenc}           
\usepackage[arrow, matrix, curve]{xy}    
\usepackage[english]{babel}                
\usepackage{graphics}
\usepackage{hyperref}
\usepackage{color}
\usepackage[normalem]{ulem}

\usepackage{shorttoc}

\theoremstyle{plain}
\newtheorem{theorem}{Theorem}[section]
\newtheorem{corollary}[theorem]{Corollary}
\newtheorem{proposition}[theorem]{Proposition}
\newtheorem{lemma}[theorem]{Lemma}

\newtheorem{observation}[theorem]{Observation}

\theoremstyle{definition}
\newtheorem{definition}[theorem]{Definition}
\newtheorem{remark}[theorem]{Remark}
\newtheorem{example}[theorem]{Example}
\newtheorem{examples}[theorem]{Examples}

\newcommand\sSet {\textnormal{sSet}}

\newcommand{\Set}{\textnormal{Set}} 
\newcommand{\op}{\mathrm{op}}

\newcommand{\Wbar}{\overline{W}} 

\newcommand{\Dec}{\mathrm{Dec}} 

\newcommand{\PSh}{\mathit{PSh}} 

\newcommand{\Maps}{\textnormal{Maps}}
\newcommand{\tfleft}{\!\xymatrix@C=1.3pc{&\ar@{->>}_\sim[l]}\!}

\newcommand{\tfright}{\!\xymatrix@C=1.3pc{\ar@{->>}^\sim[r]&}\!}

\newcommand{\RR}{\mathbb{R}}

\def\endofproof{\hfill{$\square$}\\}

\def\proofoftheorem #1 {{Proof of theorem \ref{#1}.}\hspace{7pt}}

\title{Principal $\infty$-bundles -- Presentations}
\author{Thomas Nikolaus, Urs Schreiber, Danny Stevenson}
\date{\today}

\begin{document}

\maketitle

\begin{abstract}
 We discuss two aspects of the presentation of the theory of principal $\infty$-bundles in an 
 $\infty$-topos, introduced in \cite{NSSa}, in terms of categories of simplicial (pre)sheaves.
 
 First we show that
 over a \emph{cohesive site} $C$ and for $G$ a presheaf of simplicial groups
 which is \emph{$C$-acyclic}, $G$-principal $\infty$-bundles over any
 object in the $\infty$-topos over $C$ are
 classified by hyper-{\v C}ech-cohomology with coefficients in $G$.
 Then we show that over a site $C$ with enough points, 
 principal $\infty$-bundles in the $\infty$-topos are presented
 by ordinary simplicial bundles in the sheaf topos that 
 satisfy principality by stalkwise weak equivalences.
 Finally we discuss explicit details of these presentations for the 
 discrete site (in discrete $\infty$-groupoids) and the smooth site
 (in smooth $\infty$-groupoids, generalizing Lie groupoids and differentiable stacks).
 
 In the companion article \cite{NSSc} we use these presentations for
 constructing classes of examples of (twisted) principal $\infty$-bundles and
 for the discussion of various applications.

\end{abstract}

 \newpage
 
 \tableofcontents
 
\newpage

\section{Overview}

In \cite{NSSa} we have described a general theory of 
geometric principal $\infty$-bundles (possibly twisted by local coefficients)
and their classification by (twisted) nonabelian cohomology in $\infty$-toposes.
A certain charm of this theory is that, formulated the way it is in the  
abstract language of $\infty$-topos theory, it is not only more general
but also more elegant than the 
traditional theory. For instance every $\infty$-group action is principal over its
homotopy quotient, the quotient map is automatically locally trivial,
the principal $\infty$-bundle corresponding to a classifying map
is simply its homotopy fiber (hence the universal principal $\infty$-bundle is the point), 
and the fact that all principal
$\infty$-bundles arise this way is a fairly direct consequence of the axioms 
that characterize $\infty$-toposes in the first place: the Giraud-Rezk-Lurie axioms.

While this abstract formulation provides a useful means to reason about general properties of 
principal $\infty$-bundles, it is desireable to complement this with explicit
\emph{presentations} of the structures involved (notably of $\infty$-groups,
of $\infty$-actions and of principal $\infty$-bundles) by \emph{generators and relations}.
This is typically the way that explicit examples are constructed and in terms of
which properties of these specific examples are computed in applications.

In recent years it has been well understood that the method of choice for 
presenting $\infty$-categories by generators and relations is the homotopical 
category theory of categories of simplicial presheaves, i.e.\  presheaves of simplicial sets. 
The techniques themselves have a long history, dating back to 
work of Illusie \cite{Illusie}, continued in the 
foundational work of \cite{Brown} and developed further in \cite{Joyal-lett,Jardine}, which will play a prominent role below.
Their interpretation as a \emph{generators and relations presentation}
for homotopy theoretic structures
has been amplified in the exposition of \cite{DuggerSheavesAndHomotopy}, 
and was formalized 
in terms of model category theory by the main theorem in \cite{Dugger}.
Finally \cite{Lurie} has provided the general abstract essence of this theorem
in terms of the notion of \emph{presentable $\infty$-categories}.
This is the notion of presentation that we are concerned with here.

We formalize and prove the following statements.
\begin{enumerate}
  \item 
    Over a site $C$ with a terminal object, every $\infty$-group is presented by
	a presheaf of simplicial groups $G$. (Proposition \ref{InftyGroupsBySimplicialGroups})
  \item 
    If the ambient $\infty$-topos is locally $\infty$-connected and local
	over an \emph{$\infty$-cohesive site} $C$, and if $G$ is \emph{$C$-acyclic} 
	(Definition~\ref{CAcyclic}) 
	then $G$-principal $\infty$-bundles over any object $X$
	are classified by simplicial hyper-{\v C}ech-cohomology of $X$ with coefficients in $G$.
	In fact, the $\infty$-groupoid of geometric $G$-principal $\infty$-bundles, morphisms
	and higher homotopies between these
	is equivalent to the $\infty$-groupoid of {\v C}ech cocycles, {\v C}ech coboundaries
	and higher order coboundaries
	(Theorem \ref{CAcyclicityAndLocalFibrancy}).
	
  \item 
    If $C$ is a site with enough points, then principal $\infty$-bundles over $C$ 
	are presented by ordinary simplicial bundles in sheaves over $C$ which satisfy 
	a weakened notion of principality 
	(Theorem \ref{Classification theorem for weakly G-principal bundles}).
\end{enumerate}
The first and the third statement may be thought of as strictification results, showing that
every principal $\infty$-bundle is equivalent to one that is presented by an ordinary group
object with strict group law (not up to homotopy) acting strictly on a simplicial object. 
This makes available classical principal bundle theory as a tool for constructing and
analyzing $\infty$-bundles. The second statement provides good control over the cocycles
underlying principal $\infty$-bundles. 

In Section \ref{Models} we discuss
details of the presentations for examples of sites 
that satisfy the assumptions 1, 2 and 3 above: 
\begin{itemize}
  \item the trivial site, modelling \emph{discrete geometry};
  \item the site of smooth manifolds, modelling \emph{smooth/differential geometry}.
\end{itemize}

\medskip

The presentations of higher 
principal bundles and their interpretation as cocycles 
in non-abelian cohomology has a long history. We close this introduction 
with a short historical overview and indicate how our results 
both relate to and extend previous works of other authors.  

Following the foundational work of Giraud \cite{Giraud}, it seems 
that the first paper to consider the problem of giving a geometric 
description of non-abelian cohomology was the paper \cite{Duskin-gerbes} 
of Duskin (this paper was intended as a pre-cursor to a more 
substantial discussion, which unfortunately never materialized).  This was followed by the more comprehensive treatment 
of Breen in \cite{Breen-Grothendieck}.  This paper of Breen's is noteworthy 
in that it treats non-abelian cohomology within the natural context of the 
homotopy theory of simplicial sheaves and it also introduces the notion of 
a {\em pseudo-torseur} for a group stack; a notion which is closely related 
to our notion of weakly principal simplicial bundle.  In \cite{Ulbrich1,Ulbrich2} 
Ulbrich gave a different interpretation of Duskin's work, in particular 
introducing the notion of {\em cocycle bitorsor} which is closely related 
to Murray's later notion of \emph{bundle gerbe} \cite{Mur}.  
Joyal and Tierney in \cite{Joyal-Tierney93} 
introduced a notion of {\em pseudo-torsor} 
which is again closely related to our notion of weakly principal simplicial bundle; 
their notion of pseudo-torsor was more general than the 
corresponding notion of Breen's, since Breen restricted his attention to the case 
of 1-truncated group objects while Joyal and Tierney worked with simplicial 
groupoids.  

The mid 1990s saw a flurry of interest in interpreting geometrically the standard characteristic 
classes as higher principal bundles with structure $\infty$-group 
of the form $K(\pi,n)$ for some abelian group $\pi$; the works \cite{Bry, BryMcL1,BryMcL2,Mur} 
of Brylinski, Brylinski and McLaughlin and Murray are landmarks 
from this period.  The overarching theme of these papers is to develop and apply a Chern-Weil theory 
for `higher line bundles', in particular they focus attention on the $\infty$-groups 
$\mathbf{B}U(1)$ and $\mathbf{B}^2U(1)$.
Our aim here is less restrictive; we want to develop the 
theory of principal $\infty$-bundles as a whole.  

To continue the historical 
discussion, the 2004 thesis \cite{Bartels} gave a treatment of 1-truncated 
principal $\infty$-bundles --- \emph{principal 2-bundles} --- while \cite{Jurco} gave such a treatment from 
the point of view of bundle gerbes (we note that \cite{Jurco} appeared in pre-print form in 2005).  
In \cite{Bakovic} these constructions were generalized from structure 2-groups
to structure 2-groupoids. The gauge 2-groups of principal 2-bundles were 
studied in \cite{Wockel}.  A comprehensive account is given in \cite{NikolausWaldorf}.

Continuing in this vein, a discussion of 2-truncated principal $\infty$-bundles,
\emph{principal 3-bundles}, was given in \cite{Jurco2} in the guise of 
\emph{bundle 2-gerbes}, generalizing the abelian bundle 2-gerbes 
($\mathbf{B}^2 U(1)$-principal 3-bundles) of \cite{Stevenson}.

The work that is closest to our discussion in 
Section \ref{Principal infinity-bundles presentations} is the paper \cite{JardineLuo} 
of Jardine and Luo.  
This paper goes beyond the previous work of Breen \cite{Breen-Grothendieck} 
and Joyal and Tierney \cite{Joyal-Tierney93}; it introduces a notion of 
$G$-torsor for $G$ a group in $\mathrm{sPSh}(C)$ 
for some site $C$ 
and shows that isomorphism classes of $G$-torsors in $\mathrm{sPSh}(C)$ over 
$\ast$ are in a bijective correspondence with the set of connected components 
of $\Maps(\ast,\Wbar G)$.  The presentation that we discuss in Section 
\ref{Principal infinity-bundles presentations} is similar, differing in that 
it allows the base space to be an {\em arbitrary} simplicial presheaf and in that it
reproduces the {\em full homotopy type} of the space of cocycles, not just their connected components.

Closely related also is the discussion in \cite{RobertsStevenson, Stevenson3},
which is concerned with principal $\infty$-bundles over topological spaces and 
in particular discusses their classification by traditional classifying spaces.

In summary, our work goes beyond that of all the works cited above in two 
directions; firstly we show that our notion of weakly principal bundle suffices to 
interpret the full homotopy type of the cocycle $\infty$-groupoid, and secondly, 
we work over arbitrary bases: our base need not be just a space, it could be a 1-stack 
or even an $\infty$-stack, or differentiable versions of all of these (we 
remark that the study of bundles on differentiable stacks plays an 
important role in recent work on twisted $K$-theory \cite{FHT,Xu}).

\section{Presentations of $\infty$-toposes}

The presentations of principal $\infty$-bundles and related 
structures in an $\infty$-topos, discussed below in section \ref{structures},
builds on the presentation of the $\infty$-topos itself by categories of
simplicial (pre)sheaves.
We assume the reader to be familiar with the basics of this theory
(a good starting point is the appendix of \cite{Lurie}, a classical reference is 
\cite{DwyerKanComputations}), 
but in order to set up our notation and in order to record some statements,
needed below, which are not easily found in the literature in 
the explicit form in which we will need them, we briefly
recall some basics in section \ref{PresentationBasics}. 
In \ref{SimplicialObjectsInCoproductCompleteSites} we discuss 
a general result about the representability of 
general objects in an $\infty$-topos by simplicial objects in the site.

\subsection{By simplicial presheaves}
\label{PresentationBasics}
\label{InfinityToposPresentation}

The monoidal functor $\pi_0\colon \mathrm{sSet}\to \mathrm{Set}$ 
that sends a simplicial set to its set of connected components
induces a functor $\mathrm{Ho}\colon \mathrm{sSet}\mathrm{Cat}\to 
\mathrm{Cat}$, where $\mathrm{sSet}\mathrm{Cat}$ denotes the 
category of $\mathrm{sSet}$-enriched categories 
\cite{Kelly}.  Thus if $\mathcal{C}$ is an $\mathrm{sSet}$-enriched category 
then $\mathrm{Ho}(\mathcal{C})$ is the category with the same 
underlying objects as $\mathcal{C}$ and with 
$\mathrm{Ho}(\mathcal{C})(X,Y) := \pi_0 \mathcal{C}(X,Y)$ 
for all objects
$X, Y \in \mathcal{C}$. 
An  $\mathrm{sSet}$-enriched functor
$f : \mathcal{C} \to \mathcal{D}$ is called a \emph{DK-equivalence}
if $\mathrm{Ho}(f)$ is essentially surjective and if for all $X,Y \in \mathcal{C}$
the morphism $f_{X,Y} : \mathcal{C}(X,Y) \to \mathcal{D}(f(X), f(Y))$ is a
weak homotopy equivalence. Write $W_{\mathrm{DK}} \subset \mathrm{sSet}\mathrm{Cat}$
for the inclusion of the full sub-category whose morphisms are DK-equivalences.
This is a \emph{wide subcategory}: an inclusion of categories that is bijective on objects.

For $D$ a category and $W \subset D$ a wide subcategory,
to be called the subcategory of \emph{weak equivalences}, 
the \emph{simplicial localization} $L_W D$ is the universal
$\mathrm{sSet}$-enriched category 
with the property that morphisms in $W \subset D$ become homotopy
equivalences in $L_W D$ \cite{DwyerKanComputations}. 
For $X, Y \in C$ two objects, the Kan complex
$L_W D(X,Y)$ is called the \emph{derived hom-space} or \emph{derived function complex}
or \emph{hom-$\infty$-groupoid} between these objects, in $L_W D$.

We write 
$$
  \mathrm{Grpd}_\infty := L_{W_{\mathrm{wh}}} \mathrm{sSet}
$$
for the simplicial localization of the category of simplicial sets at the 
simplicial weak homotopy equivalences, and we write
$$
  \mathrm{Cat}_\infty := L_{W_{\mathrm{DK}}} \mathrm{sSet}\mathrm{Cat}
$$
for the simplicial localization of the category of simplicial categories at the
Dwyer-Kan equivalences (both $\mathrm{Grpd}_{\infty}$ and 
$\mathrm{Cat}_{\infty}$ are large $\mathrm{sSet}$-categories). 

If a wide subcategory $W \subset D$ on a category $D$ extends to the structure of an 
$\mathrm{sSet}_{\mathrm{Quillen}}$-enriched model category on $D$ in 
the sense that $W$ coincides with the class of weak equivalences for the model structure, 
then the
full subcategory $D^\circ$ on the fibrant and cofibrant objects is enriched
in Kan complexes and DK-equivalent to its simplicial localization:
$$
  D^\circ \simeq L_W D \;\;\; \in \mathrm{Cat}_\infty
  \,.
$$

We write $\mathrm{sSet}_{\mathrm{Quillen}}$ for the standard model category structure
on simplicial sets, whose weak equivalences are the simplicial weak homotopy equivalences
$W_{\mathrm{wh}}$ and whose fibrations are the Kan fibrations. 
Then $\mathrm{Grpd}_\infty \simeq \mathrm{KanCplx} = (\mathrm{sSet}_{\mathrm{Quillen}})^\circ$.

For $C$ any category, there is a model structure 
$[C^{\mathrm{op}}, \mathrm{sSet}]_{\mathrm{proj}}$
on the category of simplicial presheaves over $C$ 
(the {\em projective} model structure), whose weak equivalences and fibrations are
those transformations that are objectwise so in $\mathrm{sSet}_{\mathrm{Quillen}}$.
If $C$ is equipped with the structure of a \emph{site} given by a (pre)topology, 
then there are corresponding localizations of the simplicial presheaves. We are
interested here in the case that $C$ has enough points.
\begin{definition}
\label{def:enough points}
  A site $C$ has \emph{enough points} if 
  a morphism $(A \stackrel{f}{\to} B)\in \mathrm{Sh}(C)$ in its sheaf topos
  is an isomorphism precisely if for every \emph{topos point}, hence for
  every geometric morphism
  $$
    (x^* \dashv x_*) : 
    \xymatrix{
      \mathrm{Set}
      \ar@<+4pt>@{<-}[r]^{x^*}
      \ar@<-4pt>[r]_{x_*}
      &
      \mathrm{Sh}(C)
    }
  $$
  from the topos $\mathrm{Set}$ of sets
  we have that $x^*(f) : x^* A \to x^* B$ is an isomorphism.
  \label{site with enough points}
\end{definition}
  Notice here that, by definition of geometric morphism, the 
  functor $i^*$ is left adjoint to $i_*$ -- hence
  preserves all colimits --
  and in addition preserves all \emph{finite} limits. 
\begin{example}
  The following sites have enough points.
  \begin{itemize}
    \item
       The categories $\mathrm{Mfd}$ ($\mathrm{SmoothMfd}$) of 
      (smooth) finite-dimensional, paracompact manifolds
       and smooth functions between them;
     \item
       the category $\mathrm{CartSp}$ of Cartesian spaces 
       $\mathbb{R}^n$ for $n \in \mathbb{N}$ and continuous (smooth) functions
       between them.
  \end{itemize}
\end{example}
These examples are discussed in more depth in \cite{NSSc} --- we refer the 
reader there for further details.
We restrict from now on attention to the case that $C$ has enough points.

A \emph{$C$-local weak equivalence} in the category
$[C^{\mathrm{op}}, \mathrm{sSet}]$ of simplicial presheaves is a natural transformation
which is \emph{stalkwise} a weak homotopy equivalence of simplicial sets. 
Let $W_C\subset [C^\op,\mathrm{sSet}]$ denote the wide sub-category 
of $C$-local weak equivalences. The simplicial 
localization
$$
  \mathrm{Sh}_\infty(C) := L_{W_C} [C^{\mathrm{op}}, \mathrm{sSet}] \;\;\; \in \mathrm{Cat}_\infty
$$
is the \emph{hypercompletion} of the \emph{$\infty$-topos} of \emph{$\infty$-sheaves} or of
\emph{$\infty$-stacks} over $C$. 

This is the statement of Proposition 6.5.2.14 of \cite{Lurie} 
together with Theorem 17 in \cite{JardineBoolean}, which gives a refinement of the
above weak equivalences to the \emph{local injective model structure} 
$[C^{\mathrm{op}}, \mathrm{sSet}]_{\mathrm{inj}, \mathrm{loc}}$ whose cofibrations are the objectwise simplicial weak equivalences. We will be interested here instead in the 
\emph{local projective model structure} $[C^{\mathrm{op}}, \mathrm{sSet}]_{\mathrm{proj},\mathrm{loc}}$
obtained as the left Bousfield localization of $[C^{\mathrm{op}}, \mathrm{sSet}]_{\mathrm{proj}}$
at the covering sieve inclusions. For the \emph{cohesive sites} $C$ considered
in Definition \ref{CohesiveSite} below this localization will already be hypercomplete and
hence we obtain the above $\infty$-topos equivalently as
$$
  \mathrm{Sh}_\infty(C) \simeq ([C^{\mathrm{op}}, \mathrm{sSet}]_{\mathrm{proj},\mathrm{loc}})^\circ
  \,.
$$

\subsection{By simplicial objects in the site}
\label{SimplicialObjectsInCoproductCompleteSites}

Sometimes it is considered desireable to present an $\infty$-stack
by a simplicial presheaf which in turn is presented by a simplicial 
object in the underlying site. We observe here that this is always possible
provided the site has arbitrary coproducts. 

\begin{definition}
  Let $C$ be a small site with enough points. Write
  $\bar C \subset [C^{\mathrm{op}}, \mathrm{sSet}]$
  for the free coproduct completion.
  Let $(\bar C^{\Delta^{\mathrm{op}}}, W)$
  be the category of simplicial objects in $\bar C$ equipped
  with the stalkwise weak equivalences inherited from the
  canonical embedding
  $$
    i 
    :
    \bar C^{\Delta^{\mathrm{op}}}
    \hookrightarrow
    [C^{\mathrm{op}}, \mathrm{sSet}]
    \,.
  $$
\end{definition}

\begin{example}
  Let $C$ be a category of \emph{connected} topological
  spaces with given extra structure and properties 
  (for instance smooth manifolds). Then 
  $\bar C$ is the category of all such spaces
 (with arbitrary many connected components). 
\end{example} 

\begin{proposition}
  \label{ModelStructureModelsSimplicialLocalization}
  The induced $\infty$-functor
  $$
     L_{W_C} \bar C^{\Delta^{\mathrm{op}}}
    \to 
    L_{W_C} [C^{\mathrm{op}}, \mathrm{sSet}]_{\mathrm{proj},\mathrm{loc}}
  $$
  is an equivalence of $\infty$-categories.
\end{proposition}
We will prove this shortly, after we have made the 
following observation.  
\begin{proposition}
  \label{DegreewiseRepresentability}
  Let $C$ be a category and $\bar C$ its free coproduct
  completion.  Then the following statements are true:
  \begin{enumerate}
 \item  Every simplicial presheaf over $C$ is equivalent in 
  $[C^{\mathrm{op}}, \mathrm{sSet}]_{\mathrm{proj}}$ to 
  a simplicial object in $\bar C$ under the image of the degreewise Yoneda
  embedding $j : \bar C^{\Delta^{\mathrm{op}}}
\to [C^{\mathrm{op}}, \mathrm{sSet}] $.

 \item  If moreover $C$ has pullbacks and sequential colimits, then the simplicial object 
  in $\bar C$  can be taken to be globally Kan, hence fibrant in 
  $[C^{\mathrm{op}}, \mathrm{sSet}]_{\mathrm{proj}}$.
  \end{enumerate}
\end{proposition}
This proposition 
can be interpreted as follows: every $\infty$-stack over
  $C$ has a presentation by a simplicial object in $\bar C$.
  Moreover this is true with respect to {\em any} Grothendieck topology 
  on $C$, since
  the weak equivalences in the global projective model
  structure remain weak equivalences in any left Bousfield localization.
 If moreover $C$ has all pullbacks (for instance for 
  topological spaces, but not for smooth manifolds) then
  every $\infty$-stack over $C$ even has a presentation by a
  globally Kan simplicial object in $\bar C$.

\proof
  The first statement is Proposition 2.8 in \cite{Dugger}, which
  says that for every $X \in [C^{\mathrm{op}}, \mathrm{sSet}]$
  the canonical morphism $QX\to X$ is a global weak equivalence.  
  Here $QX$ is the simplicial presheaf defined by the formula 
  $$
    (Q X) : [k] \mapsto \coprod_{U_0 \to \cdots \to U_k \to X_k} j(U_0)
	\,,
  $$  
  where the coproduct runs over all sequences of morphisms between representables
  $U_i$ as indicated and with the evident face and degeneracy maps. 
   The second statement follows by postcomposing with Kan's
  fibrant replacement functor (see for instance section 3 in 
  \cite{Jardine})
  $$
    \mathrm{Ex}^\infty : \mathrm{sSet} \to \mathrm{KanCplx} \hookrightarrow
	\mathrm{sSet}.	
  $$
  This functor forms new simplices by subdivision, which only involves
  forming iterated pullbacks over the spaces of the original simplices. 
\endofproof

 \begin{proof}[Proof of Proposition~\ref{ModelStructureModelsSimplicialLocalization}]
  Let 
   $Q : [C^{\mathrm{op}}, \mathrm{sSet}] 
    \to     
  \bar C^{\Delta^{\mathrm{op}}}$
  be Dugger's functor from the proof
  of Proposition~\ref{DegreewiseRepresentability}.  
  In \cite{Dugger} it is shown that for all $X$ the simplicial 
  presheaf  $Q X$
  is cofibrant in $[C^{\mathrm{op}}, \mathrm{sSet}]_{\mathrm{proj}}$
  and that the natural morphism $Q X \to X$ is a weak equivalence 
  (as we have observed previously).
  Since left Bousfield localization does not affect the cofibrations
  and only enlarges the weak equivalences, the same is still true in 
  $[C^{\mathrm{op}}, \mathrm{sSet}]_{\mathrm{proj}, \mathrm{loc}}$.
    
  Therefore we have a natural transformation
  $$
    i \circ Q \to \mathrm{Id} 
    : 
   [C^{\mathrm{op}}, \mathrm{sSet}]
   \to
   [C^{\mathrm{op}}, \mathrm{sSet}]
  $$
  whose components are weak equivalences. From this the claim
  of Proposition~\ref{ModelStructureModelsSimplicialLocalization} 
  follows by Proposition~3.5 in \cite{DwyerKanComputations}.
\end{proof}

\begin{remark}
  \index{$\infty$-topos theory!presentation by simplicial $C$-manifolds}
  \index{$\infty$-topos!presentation by simplicial $C$-manifolds}
  If the site $C$ is moreover equipped with the structure
  of a \emph{geometry} as in \cite{LurieSpaces} then there is 
 a canonical notion of a \emph{$C$-manifold}:
  a sheaf on $C$ that is \emph{locally} isomorphic to a
  representable in $C$. Write $C \mathrm{Mfd}$
  for the full subcategory of the category of presheaves on the $C$-manifolds.
 
  Then Proposition \ref{DegreewiseRepresentability} applies to 
  the category $C\mathrm{Mfd}^{\Delta^{\mathrm{op}}}$ of
  simplicial $C$-manifolds. Therefore we find that
  the $\infty$-topos over $C$ is presented by the
  simplicial localization of simplicial $C$-manifolds at
  the stalkwise weak equivalences:
  $$
    \mathrm{Sh}_\infty(C) \simeq L_{W_C} C \mathrm{Mfd}^{\Delta^{\mathrm{op}}}
    \,.
  $$
\end{remark}
\begin{example}
  \label{SmoothInfinityGroupoidByManifolds}
  \index{smooth $\infty$-groupoid!presentation by simplicial smooth manifolds}
  Let $C = \mathrm{CartSp}_{\mathrm{smooth}}$ be the 
  full subcategory of the category $\mathrm{SmthMfd}$ of
  smooth manifolds on the Cartesian spaces, $\mathbb{R}^n$,
for $n \in \mathbb{R}$. Then $\bar C \subset \mathrm{SmthMfd}$
is the full subcategory on manifolds that  are disjoint unions of
Cartesian spaces and $C \mathrm{Mfd} \simeq \mathrm{SmthMfd}$.
Therefore we have an equivalence of $\infty$-categories
$$
  \mathrm{Sh}_\infty(\mathrm{SmthMfd})
    \simeq
  \mathrm{Sh}_\infty(\mathrm{CartSp})
   \simeq
  L_{W_C}  \; \mathrm{SmthMfd}^{\Delta^{\mathrm{op}}}
  \,.
$$
\end{example}

\begin{remark}
  While the above gives fairly general conditions on a site $C$ under which 
  every $\infty$-stack is presented by a simplicial object in the site,
  and in fact by a simplicial object which is cofibrant in the 
  projective model structure on the simplicial presheaves over the site,
  this simplicial object is in general not fibrant in that model structure,
  nor will it be stalkwise fibrant in general. 
  
  In parts of the literature special attention is paid to 
  $\infty$-stacks (or just stacks) that admit a presentation by a simplicial 
  presheaf which is both: 1. represented by a simplicial object in the site
  and 2. stalkwise Kan fibrant in a suitable sense.
  (For instance Schommer-Pries discusses this for 1-stacks on manifolds and
  \cite{Wolfson} (see there for further references) for $\infty$-stacks on manifolds.) 
  It is an interesting question -- which is open at the time of this writing --  
  what these conditions on the presentation of 
  an $\infty$-stack mean intrinsically, for instance if they can be interpreted 
  as ensuring an abstract geometricity condition on an $\infty$-stack, such as
  considered for instance in \cite{LurieSpaces}.
  \label{LocallyKanAndSimplicialInSite}
\end{remark}

\section{Presentation of structures in an $\infty$-topos}
\label{structures}

In the companion article \cite{NSSa} we considered a list of structures
present in any $\infty$-topos, which form the fabric for our discussion
of principal (and associated/twisted) $\infty$-bundles. 
Here we go through the same list of notions and 
discuss aspects of their
\emph{presentation} in categories of simplicial (pre)sheaves.

\subsection{Cones}

\begin{proposition}
  \label{ConstructionOfHomotopyLimits}
  Let $A \to C \leftarrow B$ be a cospan diagram in a model category.
  Sufficient conditions for the ordinary pullback $A \times_C B$ to be 
  a homotopy pullback are
  \begin{itemize}
    \item 
	   one of the two morphisms is a fibration and all three objects are fibrant;
	\item 
	   one of the two morphisms is a fibration and the model structure is right proper.
  \end{itemize}
\end{proposition}
This appears for instance as Proposition A.2.4.4 in \cite{Lurie}.

\begin{proposition}
  A finite homotopy limit computed in 
  $[C^{\mathrm{op}}, \mathrm{sSet}]_{\mathrm{proj}}$
  presents also the homotopy limit in 
  $[C^{\mathrm{op}}, \mathrm{sSet}]_{\mathrm{proj},\mathrm{loc}}$.
 \label{FiniteHomotopyLimitsInPresheaves}
\end{proposition}

\begin{proposition}
  For $C$ a model category and $X \in C$ any object, the slice category
  $C_{/X}$ inherits a model category structure transferred along the forgetful
  functor $C_{/X} \to C$. If $X$ is fibrant in $C$, then $C_{/X}$ presents
  the slice of the $\infty$-category presented by $C$:
  $$
    \left(
	   C_{/X}
	\right)^\circ
	\simeq
	(C ^\circ)_{/X}
	\,.
  $$
  \label{SliceModelPresentsInfinitySlicing}
  \label{SliceModelStructure}
\end{proposition}

\subsection{Effective epimorphisms}

We  discuss  aspects of the presentation of
effective epimorphisms in an $\infty$-topos.  
We begin with the following observation.  

\medskip

\begin{observation}
  \label{CanonicalAtlasOfSimplicialPresheaf}
  If the $\infty$-topos $\mathbf{H}$ is presented by a category of
  simplicial presheaves, Section~\ref{InfinityToposPresentation},
  then for $X$ a simplicial presheaf, the canonical morphism
    $\mathrm{const} X_0 \to X$ in $[C^{\mathrm{op}},\mathrm{sSet}]$ 
    that includes the presheaf of 
  0-cells as a simplicially constant simplicial presheaf
  presents an effective epimorphism in $\mathbf{H}$.
\end{observation}
 This follows with Proposition 7.2.1.14 in \cite{Lurie}.
\begin{remark}
  In practice the presentation of an $\infty$-stack by a simplicial presheaf is
  often taken to be understood, and then Observation~\ref{CanonicalAtlasOfSimplicialPresheaf}
  induces also a canonical {\em atlas}, i.e.\ $\mathrm{const}X_0 \to X$. 
\end{remark}

We now discuss a fibration resolution of the canonical atlas.
 Write $\Delta_a$ for the \emph{augmented simplex category},
  which is the simplex category $\Delta$ with an initial object $[-1]$ (the empty set) adjoined.  
  The operation of {\em ordinal sum} 
  \[
  [k], [l] \mapsto [k+l +1]
  \]
  equips $\Delta_a$ with the structure of a symmetric monoidal category 
  with unit $[-1]$ (see for instance \cite{CWM}).  
  Write
  \[
    \sigma : \Delta \times \Delta \to \Delta
  \]
  for the restriction of this tensor product along the canonical 
  inclusion $\Delta \subset \Delta_a$.   

\begin{definition}
  \label{Decalage}
  Write
  $$
    \mathrm{Dec}_0 : \mathrm{sSet} \to \mathrm{sSet}
  $$
  for the functor given by precomposition with $\sigma(-,[0]) : \Delta \to \Delta$.
  This is called the plain \emph{d{\'e}calage functor} or \emph{shifting functor}.
\end{definition}
This functor was introduced in \cite{Illusie2}. A discussion in the present 
context can be found in section 2.2 of \cite{Stevenson2}, amongst 
other references.
\begin{proposition}
  \label{MorphismsOutOfPlainDecalage}
  Let $X$ be a simplicial set.  Then $\mathrm{Dec}_0X$ is 
  isomorphic to the simplicial set 
  \[
  [n]\mapsto \mathrm{Hom}( \Delta[n] \star \Delta[0], X),
  \]
  where $(-)\star (-) : \mathrm{sSet} \times \mathrm{sSet} \to \mathrm{sSet}$
  is the {\em join} of simplicial sets. 
  The canonical inclusions $\Delta[n], \Delta[0] \subset \Delta[n] \star \Delta[0]$
  induce morphisms
  $$
    \xymatrix{
	   \mathrm{Dec}_0 X \ar[d]^\simeq \ar@{->>}[r] & X
	   \\
	   \mathrm{const} X_0
	}
	\,,
  $$
  where 
  \begin{itemize}
    \item the horizontal morphism is given in degree $n$ by $d_{n+1} : X_{n+1} \to X_n$;
	\item the horizontal morphism is a Kan fibration if $X$ is a Kan complex;
	\item the vertical morphism is a simplicial deformation retraction, 
	in particular a weak homotopy equivalence. 
	\end{itemize}
\end{proposition}
\proof
  The relation to the join of simplicial sets is clear (the point being that  
  the nerve functor sends joins of categories to joins of 
  simplicial sets). The deformation retraction is classical 
  and can be found in many sources.
  To see that $\mathrm{Dec}_0 X \to X$ is 
  a Kan fibration, using the fact that 
   $(\mathrm{Dec}_0 X)_n = \mathrm{Hom}(\Delta[n] \star \Delta[0], X)$ 
   for any $n\in \mathbb{N}$, we see that the lifting problem for the diagram 
  $$
    \xymatrix{
	  \Lambda^i[n] \ar[r] \ar[d] & \mathrm{Dec}_0 X \ar[d]
	  \\
	  \Delta[n] \ar[r] & X
	}
  $$
 has a solution if and only if it is the lifting problem for the diagram 
  $$
   \raisebox{20pt}{
    \xymatrix{
	  (\Lambda^i[n] \star \Delta[n] ) \coprod_{\Lambda^i[n]} \Delta[n]
	  \ar[r]
	  \ar[d]
	  &
	  X
	  \ar[d]
	  \\
	  \Delta[n] \star \Delta[0]
	  \ar[r]
	  &
	  \ast 
	}
	}
  $$
  has a solution.  
  Here the left hand vertical 
  morphism is an anodyne morphism --- in fact an inclusion of an $(n+1)$-horn. 
  Hence a lift exists if $X$ is a Kan complex. 
  (Alternatively, one may argue by observing that $\mathrm{Dec}_0X$ 
  is the disjoint union of slices $X_{/x}$ for $x\in X_0$, and it is known 
  that $X_{/x}\to X$ is a Kan fibration if $X$ is a Kan complex --- see for 
  instance \cite{Lurie}).  
   \endofproof
\begin{corollary}
  \label{DecalageIsFibrationResolution}
 For $X$ in $[C^{\mathrm{op}}, \mathrm{sSet}]_{\mathrm{proj}}$ 
 fibrant, a fibration resolution of the 
 canonical effective epimorphism $\mathrm{const} X_0 \to X$
 from Observation~\ref{CanonicalAtlasOfSimplicialPresheaf}
 is given by the d{\'e}calage morphism
 $\mathrm{Dec}_0 X \to X$, Proposition~\ref{MorphismsOutOfPlainDecalage}.
\end{corollary}
\proof
  It only remains to observe that we have a commuting diagram
  $$
    \raisebox{20pt}{
    \xymatrix{
	  \mathrm{const}X_0 \ar[r]^s \ar[d] & \mathrm{Dec}_0 X \ar[d]
      \\
      X \ar[r]^= & X	  
	}
	}
  $$
  where the top morphism, given degreewise by the degeneracy maps in $X$,
  is a weak homotopy equivalence by classical results. 
\endofproof

\subsection{Connected objects}
\label{ConnectedObjects}
\label{Connected objects presentations}

In every $\infty$-topos $\mathbf{H}$ there is a notion of \emph{connected objects},
which form the objects of the full sub-$\infty$-category $\mathbf{H}_{\geq 1}$.
We discuss here presentations of connected and of \emph{pointed}
connected objects in $\mathbf{H}$ 
by means of presheaves of pointed or reduced simplicial sets.

\medskip

\begin{observation}
  Under the presentation 
  $\mathrm{Grpd}_{\infty} \simeq (\mathrm{sSet}_{\mathrm{Quillen}})^\circ$,
  a Kan complex $X \in \mathrm{sSet}$ presents an $n$-connected $\infty$-groupoid
  precisely if
  \begin{enumerate}  
    \item $X$ is inhabited (not empty);
	\item all simplicial homotopy groups $\pi_k(X)$ of $X$ in degree $k \leq n$ are trivial.
  \end{enumerate}
\end{observation}
\begin{definition}
  For $n \in \mathbb{N}$ a simplicial set $X \in \mathrm{sSet}$
  is \emph{$n$-reduced} if its $n$-skeleton is the point
  $$
    \mathrm{sk}_n X = *
	\,,
  $$
  in other words, if it has a single $k$-simplex for all $k \leq n$.
  For \emph{0-reduced} we also just say \emph{reduced}. 
  Write
  $$
    \mathrm{sSet}_n \hookrightarrow \mathrm{sSet}
  $$
  for the full subcategory of $n$-reduced simplicial sets.
\end{definition}
\begin{proposition}
  \label{PropertiesOfTheReducedsSetInclusion}
  The $n$-reduced simplicial sets form a reflective subcategory
  $$
    \xymatrix{
	  \mathrm{sSet}_n
	  \ar@{<-}@<+4pt>[r]^{\mathrm{red}_n}
	  \ar@{^{(}->}@<-4pt>[r]
	  &
	  \mathrm{sSet}
	}
  $$
  of the category of simplicial sets, 
  with the reflector $\mathrm{red}_n$ given on a simplicial set $X$ by
  $\mathrm{red}_n(X) = X/\mathrm{sk}_n X$, 
  in other words it
  identifies all the $(k \leq n)$-vertices
  of $X$.
  
  The inclusion $\mathrm{sSet}_n \hookrightarrow \mathrm{sSet}$ 
  uniquely factors through the forgetful functor $\mathrm{sSet}^{*/} \to \mathrm{sSet}$
  from pointed simplicial sets, and that factorization is co-reflective
  $$
    \xymatrix{
	  \mathrm{sSet}_n
	  \ar@{^{(}->}@<+4pt>[r]^{}
	  \ar@{<-}@<-4pt>[r]_{E_{n+1}}
	  &
	  \mathrm{sSet}^{*/}
	}
	\,.
  $$
  Here the co-reflector $E_{n+1}$ sends a pointed simplicial 
  set $* \stackrel{x}{\to} X$ to the sub-object $E_{n+1}(X,x)$,
  the $(n+1)$st Eilenberg subcomplex of the pointed simplicial set $X$.
\end{proposition}
\begin{remark}
  Recall, see for instance Definition 8.3  in \cite{May}, that for a 
  pointed simplicial set $\ast\xrightarrow{x} X$, the simplicial set
  $E_{n+1}(X,x)$ is the subcomplex of $X$ consisting 
  of cells whose $n$-faces coincide with the base point, hence is the fiber
  \[
   \xymatrix{ 
     E_{n+1}(X,x) \ar[r] \ar[d] & X \ar[d] 
	  \\ 
     \{\ast\} \ar[r] & \mathrm{cosk}_n X } 
  \]
  of the projection to the $n$-coskeleton $\mathrm{cosk}_nX$.
  
  For $(* \to X) \in \mathrm{sSet}^{*/}$ such that $X \in \mathrm{sSet}$ is 
  Kan fibrant and $n$-connected,
  the counit $E_{n+1}(X,*) \to X$ is a homotopy equivalence.
  This statement appears for instance as part of Theorem 8.4 in \cite{May}.
\end{remark}
\begin{proposition}
  \label{PresentationOfPointedConnectedObjectsByPresheavesOfReducedsSets}
  Let $C$ be a site with a terminal object and let 
  $\mathbf{H} := \mathrm{Sh}_\infty(C)$.
  Then under the presentation
  $\mathbf{H} \simeq ([C^{\mathrm{op}}, \mathrm{sSet}]_{\mathrm{proj}, \mathrm{loc}})^\circ$ 
  every pointed $n$-connected object in $\mathbf{H}$ is presented by a
  presheaf of $n$-reduced simplicial sets, under the canonical inclusion
  $[C^{\mathrm{op}}, \mathrm{sSet}_n]
	\hookrightarrow
    [C^{\mathrm{op}}, \mathrm{sSet}]$.
\end{proposition}
\proof
  Let $X \in [C^{\mathrm{op}}, \mathrm{sSet}]$ be a simplicial presheaf presenting
  the given pointed, connected object. 
  Then its objectwise Kan fibrant replacement
  $\mathrm{Ex}^\infty X$ is still a presentation, fibrant in the 
  global projective model structure. 
  Since the terminal object in $\mathbf{H}$ is presented by the 
  terminal simplicial presheaf and since by assumption on $C$
  this is representable
  and hence cofibrant in the projective model structure, the point inclusion is
  presented by a morphism of simplicial presheaves $* \to \mathrm{Ex}^\infty X$,
  hence by a presheaf of pointed simplicial sets 
  $(* \to \mathrm{Ex}^\infty X) \in [C^{\mathrm{op}}, \mathrm{sSet}^{*/}]$. 
  So with Proposition
  \ref{PropertiesOfTheReducedsSetInclusion}
  we obtain the presheaf of $n$-reduced simplicial sets
  $$
    E_{n+1}(\mathrm{Ex}^\infty X, *) 
	 \in 
	[C^{\mathrm{op}}, \mathrm{sSet}_n] 
	  \hookrightarrow
	[C^{\mathrm{op}}, \mathrm{sSet}]
  $$
  and the inclusion $E_{n+1}(\mathrm{Ex}^\infty X,*) \to \mathrm{Ex}^\infty X$
  is a global weak equivalence, hence a local weak equivalence, hence
  exhibits $E_{n+1}(\mathrm{Ex}^\infty X,*)$ as another presentation of the
  object in question.
\endofproof

We next describe a slightly enhanced version of the model structure on reduced simplicial sets 
introduced by Quillen in \cite{QuillenRHT}.   
\begin{proposition}
  \label{ModelStructureOnReducedSimplicialSets}
  The category $\mathrm{sSet}_{0}$ of reduced simplicial sets 
  carries a left proper combinatorial model category structure
  whose weak equivalences and cofibrations are those in $\mathrm{sSet}_{\mathrm{Quillen}}$
  under the inclusion $\mathrm{sSet}_{0} \hookrightarrow \mathrm{sSet}$.
\end{proposition}
\proof
 This enhanced version of the classical theorem from \cite{QuillenRHT} follows from 
  Proposition A.2.6.13 in \cite{Lurie}, taking the set $C_0$ there to be
  $$
    C_0 := \{ \mathrm{red}(\Lambda^k[n] \to \Delta[n])\}_{n \in \mathbb{N}, 0 \leq k \leq n}
	\,,
  $$
  the image 
  of the generating cofibrations in 
  $\mathrm{sSet}_{\mathrm{Quillen}}$ 
  under the left adjoint $\mathrm{red}$ to the inclusion 
  functor (Proposition~\ref{PropertiesOfTheReducedsSetInclusion}).
\endofproof
\begin{lemma}
  \label{HomotopyPropertyOfInclusionOfReducedsSets}
  A fibration $f\colon X\to Y$ in $\mathrm{sSet}_{0}$ (for 
  the model structure of Proposition~\ref{ModelStructureOnReducedSimplicialSets}) is a 
  Kan fibration precisely if
  it has the right lifting property against the morphism
  $(* \to S^1) := \mathrm{red}(\Delta[0] \to \Delta[1])$.  In 
  particular every fibrant object in $\mathrm{sSet}_{0}$ 
  is a Kan complex.  
 \end{lemma}
\begin{proof} 
The first statement appears as V Lemma 6.6. in \cite{GoerssJardine}. 
The second (an immediate consequence) as V Corollary 6.8.
\end{proof} 

\begin{proposition}
  The adjunction
  $$
    \xymatrix{
	  \mathrm{sSet}_{0}
	  \ar@{^{(}->}@<+4pt>[r]^{i}
	  \ar@{<-}@<-4pt>[r]_{E_1}
	  &
	  \mathrm{sSet}_{\mathrm{Quillen}}^{*/}
	}
  $$
  from Proposition~\ref{PropertiesOfTheReducedsSetInclusion}
  is a Quillen adjunction between the model structure from 
  Proposition~\ref{ModelStructureOnReducedSimplicialSets}
  and the co-slice model structure, Proposition~\ref{SliceModelStructure},
  of $\mathrm{sSet}_{\mathrm{Quillen}}$ under the point. 
  This presents the 
  full inclusion
  $$
    {(\mathrm{Grpd}_\infty)}^{*/}_{\geq 1} \hookrightarrow
	{\mathrm{Grpd}_\infty}^{*/}
  $$
  of connected pointed $\infty$-groupoids into all pointed $\infty$-groupoids.
\end{proposition}
\proof
  It is clear that the inclusion $i\colon \mathrm{sSet}_0\hookrightarrow 
  \mathrm{sSet}^{\ast/}_{\mathrm{Quillen}}$ preserves cofibrations 
  and acyclic cofibrations, in fact all weak equivalences.
  Since the point is cofibrant in $\mathrm{sSet}_{\mathrm{Quillen}}$, 
  the model structure on the 
  right is by Proposition~\ref{SliceModelPresentsInfinitySlicing} 
  indeed a presentation of $\mathrm{Grpd}_{\infty}^{*/}$.
  
  We claim now that the derived $\infty$-adjunction of this 
  Quillen adjunction presents a homotopy full and faithful inclusion whose
  essential image consists of the connected pointed objects.
  To show this it is 
  sufficient to show that for the derived functors there is a natural weak equivalence
  $$
    \mathrm{id} \simeq \mathbb{R}E_1\circ \mathbb{L}i
	\,.
  $$
  This is the case, because by Proposition~\ref{HomotopyPropertyOfInclusionOfReducedsSets} 
  the composite derived functors are computed by
  the composite ordinary functors precomposed with a fibrant replacement 
  functor $P$, so that
  we have a natural morphism
  $$
    X \stackrel{\simeq}{\to} P X = E_1 \circ i (P X) \simeq 
	(\mathbb{R}E_1)\circ (\mathbb{L}i) (X)
    \,.	
  $$
  Hence $\mathbb{L} i$ is homotopy full-and faithful and 
  by Proposition~\ref{PresentationOfPointedConnectedObjectsByPresheavesOfReducedsSets}
  its essential image consists of the connected pointed objects. 
\endofproof

\subsection{Groupoids}
\label{StrucInftyGroupoids}

We discuss aspects of the presentation of \emph{groupoid objects} in
an $\infty$-sheaf topos $\mathbf{H} = \mathrm{Sh}_\infty(C)$,
notably of the realization $\infty$-functor
$$
  \varinjlim : \mathrm{Grpd}(\mathbf{H}) \to \mathbf{H}
$$
given by the $\infty$-colimit over the underlying simplicial diagram of the
groupoid object.

In \cite{BergnerInvertibleb} a presentation of groupoid objects in 
$\infty\mathrm{Grpd}$ is discussed in terms of simplicial objects in $\mathrm{sSet}_{\mathrm{Quillen}}$, 
called `invertible Segal spaces' in \cite{BergnerInvertibleb}. This has a straightforward generalization
to a presentation of groupoid objects in a sheaf $\infty$-topos $\mathrm{Sh}_\infty(C)$
by simplicial objects in a category of simplicial presheaves. 
We discuss here a presentation of 
homotopy colimits over such simplical diagrams given by the 
\emph{diagonal simplicial set} or the \emph{total simplicial set} 
associated with a bisimplicial set. This serves as the basis for the 
discussion of universal weakly principal simplicial bundles below in 
Section \ref{Universal princial bundles}.
For some general background on homotopy colimits 
the way we need them here, a good survey is \cite{Gambino}.

\medskip

\begin{proposition}
  \label{TheSimplexAndTheFatSimplex}
  Write $[\Delta, \mathrm{sSet}]$ for the category of cosimplicial simplicial sets.
  For $\mathrm{sSet}$ equipped with its cartesian monoidal structure, the
  tensor unit is the terminal object $*$.
  \begin{itemize}
    \item The \emph{simplex functor}
	  $$
	    \Delta : [n] \mapsto \Delta[n] := \Delta(-,[n])
	  $$
	  is a cofibrant resolution of $*$ in $[\Delta, \mathrm{sSet}_{\mathrm{Quillen}}]_{\mathrm{Reedy}}$;
    \item 
     the \emph{fat simplex functor}
     $$
	   \mathbf{\Delta} : [n] \mapsto N(\Delta/[n])
     $$	 
	 is a cofibrant resolution of $*$ in $[\Delta, \mathrm{sSet}_{\mathrm{Quillen}}]_{\mathrm{proj}}$.
  \end{itemize}
\end{proposition}
\begin{proposition}
  \label{BousfieldKanFormula}
  Let $C$ be a simplicial model category and $F : \Delta^{\mathrm{op}} \to C$
  a simplicial diagram
  \begin{enumerate}
    \item 
	  If every monomorphism in $C$ is a cofibration, then the homotopy colimit over
	  $F$ is given by the realization, i.e.\ 
	  $$
	    \mathbb{L}\varinjlim F \simeq \int^{[n] \in \Delta} F([n]) \cdot \Delta[n]
		\,.
	  $$
	\item
	  If $F$ takes values in cofibrant objects, then the homotopy colimit 
	  over $F$ is given by the fat realization, i.e.\ 
	  $$
	    \mathbb{L}\varinjlim F \simeq \int^{[n] \in \Delta} F([n]) \cdot \mathbf{\Delta}[n]
		\,.
	  $$
	\item If $F$ is Reedy cofibrant, then the canonical morphism
	$$
	  \int^{[n] \in \Delta} F([n]) \cdot \mathbf{\Delta}[n]
	  \to 
	  \int^{[n] \in \Delta} F([n]) \cdot \Delta[n]
	$$
	(the \emph{Bousfield-Kan map})
	is a weak equivalence.
  \end{enumerate}
\end{proposition}

\begin{proposition}
  \label{SimplicialHocolimGivenByDiagonal}
  The homotopy colimit of a simplicial diagram in $\mathrm{sSet}_{\mathrm{Quillen}}$,
  or more generally of a simplicial diagram of simplicial presheaves, is given by
  the diagonal of the corresponding bisimplicial set / bisimplicial presheaf.
  
  More precisely, for
  $$
    F : \Delta^{\mathrm{op}} \to [C^{\mathrm{op}}, \mathrm{sSet}_{\mathrm{Quillen}}]_{\mathrm{inj}, \mathrm{loc}}
  $$
  a simplicial diagram, its homotopy colimit is given by
  $$
    \mathbb{L} \varinjlim F_\bullet 
	  \simeq 
	  d F : ([n] \mapsto (F_n)_n)
	\,.
  $$
\end{proposition}
\proof
  By Proposition~\ref{BousfieldKanFormula}
  the homotopy colimit is given by the coend
  $$
    \mathbb{L}\varinjlim F_\bullet \simeq 
	\int^{[n] \in \Delta} F_n \cdot \Delta[n]
	\,.
  $$
  By a standard fact (e.g. exercise 1.6 in \cite{GoerssJardine}), 
  this coend is in fact isomorphic to the diagonal.
\endofproof

\begin{definition}
  \label{TotalSimplicialSetAndTotalDecalage}
  \label{TotalSimplicialSet}
 Let $\sigma\colon \Delta\times \Delta\to \Delta$ denote 
 ordinal sum.  Write
  $$
    \sigma^* : \mathrm{sSet} \to [\Delta^{\mathrm{op}}, \mathrm{sSet}]
  $$
  for the operation of precomposition with this functor. By right
  Kan extension this induces an adjoint pair of functors
  $$
    (\sigma_* \dashv \sigma^*)
	:
    \xymatrix{
	  [\Delta^{\mathrm{op}}, \mathrm{sSet}]
	  \ar@<+4pt>@{<-}[r]^>>>>{\sigma^*}
	  \ar@<-2pt>@{->}[r]_>>>>{\sigma_*}
	  &
	  \mathrm{sSet},
	}
  $$
  where
  \begin{itemize}
    \item $\mathrm{Dec} := \sigma^*$ is called the \emph{total d{\'e}calage} functor;
	\item $\sigma_*$ is called the \emph{total simplicial set} functor.
  \end{itemize}
\end{definition}
The total simplicial set functor was introduced in \cite{ArtinMazur2}, for 
further discussion see \cite{CegarraRemedios,Stevenson2}.
\begin{remark}
  By definition, for  $X \in \mathrm{sSet}$,
  its total d{\'e}calage is the bisimplicial set $\mathrm{Dec} X$ whose set 
  of $(k,l)$ bisimplices is given by  
  $$
    (\mathrm{Dec} X)_{k,l}  = X_{k+l+1}
	\,.
  $$
\end{remark}
\begin{remark}
  \label{Total simplicial object is built from finite limits}
  For $X \in [\Delta^{\mathrm{op}}, \mathrm{sSet}]$, the simplicial set
  $\sigma_*X$ is in each degree given by an equalizer of maps between finite products
  of components of $X$ (see for instance equation (2) of 
  \cite{Stevenson2}).  Hence forming $\sigma_*$ is compatible with sheafification
  and other processes that preserve finite limits.
\end{remark}

\begin{proposition}
  \label{TotalSimpSetEquivalentToDiagonal}
  The following statements are true: 
  \begin{itemize}
    \item 
	  for every $X \in [\Delta^{\mathrm{op}}, \mathrm{sSet}]$, the canonical morphism 
	  $$ 
	    d X \to \sigma_* X
	  $$
	  from the diagonal to the total simplicial set
	  is a weak equivalence in $\mathrm{sSet}_{\mathrm{Quillen}}$;
	\item
	  for every $X\in \mathrm{sSet}$ the adjunction unit
	  $$
	    X \to \sigma_* \sigma^* X
	  $$
	  is a weak equivalence in $\mathrm{sSet}_{\mathrm{Quillen}}$.
  \end{itemize}
  For every $X \in \mathrm{sSet}$
  \begin{itemize}
    \item 
	   there is a natural isomorphism $\sigma_* \mathrm{const} X \simeq X$.
  \end{itemize}
\end{proposition}
These statements are due to Cegarra and Remedios in \cite{CegarraRemedios} 
and independently Joyal and Tierney (unpublished) --- see also \cite{Stevenson2}.
\begin{corollary}
  \label{SimplicialHomotopyColimitByCodiagonal}
  For 
  $$
    F : \Delta^{\mathrm{op}} \to [C^{\mathrm{op}}, \mathrm{sSet}_{\mathrm{Quillen}}]_{\mathrm{inj}, \mathrm{loc}}
  $$ a simplicial object in simplicial presheaves, its homotopy colimit is given by
  applying objectwise over each $U \in C$ the total simplicial set functor $\sigma_*$, i.e.\ 
  $$
    \mathbb{L} \lim\limits_{\longrightarrow} F \simeq (U \mapsto \sigma_* F(U))
	\,.
  $$
\end{corollary}
\proof
  By Proposition~\ref{TotalSimpSetEquivalentToDiagonal} this follows from 
  Proposition~\ref{SimplicialHocolimGivenByDiagonal}.
\endofproof
\begin{remark}
  The use of the total simplicial set instead of the diagonal simplicial set in the 
  presentation of simplicial homotopy colimits is useful and 
  reduces to various traditional notions in particular
  in the context of group objects and action groupoid objects. We discuss this further in  
  Section~\ref{InfinityGroupPresentations} and Section~\ref{Universal princial bundles} below.
\end{remark}

\subsection{Groups}
\label{StrucInftyGroups}
\label{InfinityGroupPresentations}

Every $\infty$-topos $\mathbf{H}$ 
comes with a notion of \emph{$\infty$-group object} that generalizes the
ordinary notion of group object in a topos as well as that of
grouplike $A_\infty$ space in $\mathrm{Top} \simeq \mathrm{Grpd}_{\infty}$.
We discuss presentations of $\infty$-group objects by 
presheaves of simplicial groups.

\medskip

\begin{definition}
  \label{BarWAsCompositeWithTotal}
  One writes $\Wbar$ for the composite functor from simplicial
  groups to simplicial sets given by
  \[
\Wbar \colon [\Delta^{\mathrm{op}},\mathrm{Grp}] 
\xrightarrow{[\Delta^{\mathrm{op}},\mathbf{B}]} 
[\Delta^\op,\mathrm{Grpd}] 
\xrightarrow{[\Delta^{\mathrm{op}},N]} 
[\Delta^{\mathrm{op}},\mathrm{sSet}] 
\xrightarrow{\sigma_*} 
\mathrm{sSet},  
  \]
   where
   $[\Delta^{\mathrm{op}},\mathbf{B}] : [\Delta^{\mathrm{op}}, \mathrm{Grp}] \to 
  [\Delta^{\mathrm{op}}, \mathrm{Grpd}]$  is the functor from simplicial groups
  to simplicial groupoids that degreewise sends a group to the corresponding one-object 
  groupoid.
\end{definition}
This simplicial 
delooping $\Wbar$ was originally introduced in 
\cite{MacLane}. 
The above formulation is due to Duskin, see Lemma 15 in \cite{Stevenson2}.

\begin{remark}
  The functor $\overline{W}$ takes values in \emph{reduced} simplicial sets, 
  i.e.\ $\overline{W}\colon [\Delta^{\mathrm{op}},\mathrm{Grp}]\to \mathrm{sSet}_{\mathrm{red}}$.  
 \end{remark}
\begin{remark}
  For $G$ a simplicial group, the simplicial set $\Wbar G$ is,
  by Corollary~\ref{SimplicialHomotopyColimitByCodiagonal}, the
  homotopy colimit over a simplicial diagram in simplicial sets. 
  Below in \ref{Principal infinity-bundles presentations} we see that 
  this simplicial diagram is that presenting the groupoid object 
  $*/\!/G$ which is the action groupoid of $G$ acting trivially on the point.
\end{remark}
\begin{proposition}
  \label{ModelStructureOnSimplicialGroups}
  The category $\mathrm{sGrp}$ of simplicial groups carries a cofibrantly
  generated model structure for which the fibrations and the weak equivalences
  are those of $\mathrm{sSet}_{\mathrm{Quillen}}$ under the forgetful functor
  $\mathrm{sGrpd} \to \mathrm{sSet}$.
\end{proposition}
\proof
  This is originally due to \cite{Quillen}, for a more recent account see
  V Theorem 2.3 in \cite{GoerssJardine}. Note that since the model structure is
  therefore transferred along the forgetful functor, it inherits 
  generating (acyclic) cofibrations from those of $\mathrm{sSet}_{\mathrm{Quillen}}$.
\endofproof
We now consider a presentation of the 
looping/delooping equivalence $\mathrm{Grp}(\mathbf{H}) \simeq \mathbf{H}^{*/}_{\geq 1}$
due to Lurie, recalled as Theorem 2.14  in \cite{NSSa}.
\begin{theorem}[\cite{QuillenRHT}]
  \label{SimplicialLoopingQuillenEquivalence}
  The functor $\Wbar$ is the right adjoint of a Quillen equivalence
$$
  (L \dashv \Wbar) 
    :
   \xymatrix{ 
    \mathrm{sGrp} 
      \ar@<-3pt>[r]_{\Wbar}
      \ar@{<-}@<+3pt>[r]^{L}
      &
    \mathrm{sSet}_{0}
   }
$$
with respect to the model structures of Proposition~\ref{ModelStructureOnSimplicialGroups}
and Proposition~\ref{ModelStructureOnReducedSimplicialSets}.
In particular 
\begin{itemize}
  \item the adjunction unit is a weak equivalence 
$$
  Y \stackrel{\simeq}{\to} \Wbar L Y  
$$
for every reduced simplicial set $Y$, 
\item  $\Wbar G$ is a Kan complex for any simplicial group $G$.
\end{itemize}
\end{theorem}
This result is discussed for instance in chapter V of \cite{GoerssJardine};
a new proof that the unit of the adjunction is a weak equivalence is given in \cite{Stevenson2}.
\begin{definition}
  \label{WGToWbarG}
  For $G$ a simplicial group, write $WG = \mathrm{Dec}_0\overline{W}G$ (see Definition~\ref{Decalage}) and 
  write 
  $$
    W G \to \Wbar G
  $$
  for the canonical morphism $\mathrm{Dec}_0\overline{W}G\to \overline{W}G$ of 
  Corollary~\ref{DecalageIsFibrationResolution}.
\end{definition}
This morphism is the standard presentation of the 
\emph{universal $G$-principal simplicial bundle}. We discuss this further in
Section~\ref{Universal princial bundles} below.
The characterization by d{\'e}calage of the total space $W G$ 
is made fairly explicit on p. 85 of \cite{Duskin}; a fully explicit statement can be found in
\cite{RobertsStevenson}.
\begin{proposition}
  \label{WGToWbarGIsFibrationResolution}
  The morphism $W G \to \Wbar G$ is a Kan fibration resolution of the point inclusion
  ${*} \to \Wbar G$. 
\end{proposition}

\proof
  This follows directly from the characterization of $W G \to \Wbar G$ by d{\'e}calage 
  (Corollary~\ref{DecalageIsFibrationResolution}).
\endofproof

This statement appears in \cite{May} as the union of two results there: Lemma 18.2 of \cite{May} gives the
fibration property; Proposition~ 21.5 of \cite{May} gives the contractibility of $W G$. 
\begin{corollary}
  \label{GToWGToWbarGPresentsLoopingFiberSequence}
  For $G$ a simplicial group, the sequence of simplicial sets
  $$
    \xymatrix{
      G \ar[r] & W G \ar@{->>}[r] & \Wbar G
	}
  $$
  is a presentation in $\mathrm{sSet}_{\mathrm{Quillen}}$ 
  by a pullback of a Kan fibration of the looping fiber sequence
  $$
    G \to * \to \mathbf{B}G
  $$
  in $\mathrm{Grpd}_{\infty}$.
\end{corollary}
\proof
  One finds that $G$ is the 1-categorical fiber of $W G \to \Wbar G$.
  The statement then follows using Proposition~\ref{WGToWbarGIsFibrationResolution} together with  
  Proposition~\ref{ConstructionOfHomotopyLimits}.
\endofproof
The universality of $W G \to \Wbar G$ for $G$-principal simplicial bundles is the 
topic of section 21 in \cite{May}.
\begin{corollary}
  \label{SimplicialLoopingModelsInfinityLooping}
  The Quillen equivalence $(L \dashv \Wbar)$ from Theorem 
  \ref{SimplicialLoopingQuillenEquivalence} is a presentation of the looping/delooping
  equivalence $\mathrm{Grp}(\mathbf{H}) \simeq \mathbf{H}^{*/}_{\geq 1}$  
  for the $\infty$-topos $\mathbf{H} = \mathrm{Grpd}_{\infty}$.
\end{corollary}
We now lift all these statements from simplicial sets to simplicial presheaves.
\begin{proposition} 
 \label{InftyGroupsBySimplicialGroups}
If the cohesive $\infty$-topos $\mathbf{H}$ has site of definition $C$ 
with a terminal object, then 
\begin{itemize}
\item every $\infty$-group object has a presentation by a presheaf of simplicial groups 
  $$
    G \in [C^{\mathrm{op}}, \mathrm{sGrp}] \stackrel{U}{\to} [C^{\mathrm{op}}, \mathrm{sSet}]
  $$ 
  which is fibrant in $[C^{\mathrm{op}}, \mathrm{sSet}]_{\mathrm{proj}}$;
\item
   the corresponding delooping object is presented by the presheaf 
  $$
    {\Wbar} G \in [C^{\mathrm{op}}, \mathrm{sSet}_{0}] \hookrightarrow [C^{\mathrm{op}}, \mathrm{sSet}]
  $$
obtained from $G$ by applying the functor $\overline{W}$ 
objectwise.   
\end{itemize}
\end{proposition}
\proof
By the fact recalled as Theorem 2.14 in \cite{NSSa}, every $\infty$-group is the loop
space object of a pointed connected object. By 
Proposition~\ref{PresentationOfPointedConnectedObjectsByPresheavesOfReducedsSets}
every such object is presented by a presheaf $X$ of reduced simplicial sets.
By the 
simplicial looping/delooping Quillen equivalence, 
Theorem~\ref{SimplicialLoopingQuillenEquivalence},
the presheaf 
$$
  \Wbar L X \in [C^{\mathrm{op}}, \mathrm{sSet}]_{\mathrm{proj}}
$$
is objectwise weakly equivalent to the simplicial presheaf $X$.
From this the statement follows with Corollary~\ref{GToWGToWbarGPresentsLoopingFiberSequence},
combined with Proposition~\ref{FiniteHomotopyLimitsInPresheaves}, which together say that
the presheaf $L X$ of simplicial groups presents the given $\infty$-group.
\endofproof
\begin{remark}
  We may read this as saying that every $\infty$-group may be 
  \emph{strictified}.
\end{remark}

\subsection{Cohomology}
\label{section.Cohomology}
\label{GroupoidsOfMorphisms}

We discuss presentations of the \emph{hom-$\infty$-groupoids}, hence of 
\emph{cocycle $\infty$-groupoids}, hence of the cohomology in an $\infty$-topos.

We consider two roughly complementary aspects

\begin{itemize}
  \item In section \ref{CocyclesByHyperCechCohomology} we study sufficient
  conditions on a simplicial presheaf $A$ such that the ordinary 
  simplicial hom $[C^{\mathrm{op}}, \mathrm{sSet}](Y,A)$ 
  out of a split hypercover $\xymatrix{Y \ar[r] & X}$ is already the
  correct derived hom out of $X$. Since this simplicial hom is the Kan complex of 
  simplicial hyper-{\v C}ech cocycles relative to $Y$ with coefficients in $A$,
  this may be taken to be a sufficient condition for $A$-{\v C}ech cohomology to 
  produce the correct intrinsic cohomology.
  \item 
   In section \ref{Categories of fibrant objects} we consider not a full
   model category structure but just the structure of a \emph{category of fibrant objects}.
   In this case there is no notion of split hypercover and instead one has to 
   consider all possible covers and refinements between them. 
   A central result of \cite{Brown} shows that this produces the
   correct cohomology classes. Here we discuss the refinement of this classical 
   statement to the full cocycle $\infty$-groupoids.
\end{itemize}

\subsubsection{By hyper-{\v C}ech-cohomology in $C$-acyclic simplicial groups}
\label{CocyclesByHyperCechCohomology}

The condition on an object 
$X \in [C^{\mathrm{op}}, \mathrm{sSet}]_{\mathrm{proj}}$
to be fibrant models the fact that $X$ is an 
\emph{$\infty$-presheaf of $\infty$-groupoids}. 
The condition that $X$ is also fibrant as an object in
$[C^{\mathrm{op}}, \mathrm{sSet}]_{\mathrm{proj},\mathrm{loc}}$
models the higher analog of the sheaf condition: it makes $X$ an
\emph{$\infty$-sheaf}/$\infty$-stack.
For generic sites, $C$-fibrancy in the local model structure is a
property rather hard to check or establish concretely. But 
often a given site can be replaced by another site on which 
the condition is easier to control, without changing the
corresponding $\infty$-topos, up to equivalence. Here we discuss
a particularly nice class of sites called
\emph{$\infty$-cohesive sites} \cite{Schreiber}, and describe explicit conditions for
a simplicial presheaf over them to be fibrant.
\begin{definition}
  \label{CohesiveSite}
  A site $C$ is \emph{$\infty$-cohesive} if
  \begin{enumerate}
    \item 
      it has a terminal object;
    \item
  there is a generating coverage
  such that for every generating cover
  $\{U_i \to U\}$ we have 
\begin{enumerate}
  \item 
     the {\v C}ech nerve 
       $\check{C}(\{U_i\}) \in [C^{\mathrm{op}}, \mathrm{sSet}]$ 
     is degreewise a coproduct of representables;
   \item
    the limit and colimit functors, $\varprojlim\colon [C^{\mathrm{op}},\mathrm{sSet}]\to 
     \mathrm{sSet}$ and 
     $\varinjlim\colon  [C^{\mathrm{op}}, \mathrm{sSet}] \to \mathrm{sSet}$ respectively, 
     send the {\v C}ech nerve projection $\check{C}(\{U_i\})\to U$ to a weak homotopy equivalence:
     $$
       \varinjlim \check{C}(\{U_i\}) \xrightarrow{\simeq}    
        \varinjlim U = *
     $$
    and 
     $$
       \varprojlim \check{C}(\{U_i\}) \stackrel{\simeq}{\to}    
        \varprojlim U  
        \,.            
     $$
 \end{enumerate}
 \end{enumerate}
 We call the generating covers satisfying the conditions of 2 (b) the
 \emph{good covers} in $C$.
\end{definition}
\begin{remark}
  Since $C$ is assumed to have a terminal object, 
  the limit over a functor $C^{\mathrm{op}} \to \mathrm{Set}$ 
  is the evaluation on that object: 
  $$
    \varprojlim U = C(*,U)
    \,.
  $$ 
  On the other hand, the colimit of a representable $\mathrm{Set}$-valued functor
  is the singleton set: $\varinjlim U \simeq *$.
  Therefore together with the assumption that the 
  {\v C}ech nerve is degreewise representable the condition 
  $\varinjlim \check{C}(\{U_i\}) \stackrel{\simeq}{\to} \varinjlim U$
  says that the simplicial set obtained from the {\v C}ech
  nerve by replacing each $k$-fold intersection with an abstract $k$-simplex
  is contractible. 
\end{remark}
This last condition is familiar from the \emph{nerve theorem}
\cite{borsuk}:
\begin{theorem}
  Let $X$ be a paracompact topological space. Let $\{U_i \to X\}$
  be a good open cover (all non-empty $k$-fold intersections
  $U_{i_1} \cap \cdots U_{i_k}$ for $k \in \mathbb{N}$ 
  are homeomorphic to an open ball).
  Then the simplicial set
  $$
    \Pi(X)
    :=
     \mathrm{contr}
     \left(
     \int^{[k] \in \Delta} \coprod_{i_0, \cdots, i_k} U_{i_0} \cap \cdots \cap U_{i_k} 
     \right)
     =
     \int^{[k] \in \Delta} \coprod_{i_0, \cdots, i_k} * 
    \in \mathrm{sSet}
    \,,
  $$
  where $\mathrm{contr}$ is the functor that degreewise sends contractible
  spaces to points,
  is weakly homotopy equivalent to the singular simplicial set
  of $X$:
  $$
    \Pi(X) \xrightarrow{\simeq} \mathrm{Sing} X,
  $$
  and hence presents the homotopy type of $X$.
\end{theorem}
\begin{remark}
  The conditions on an $\infty$-cohesive site ensure that
  the {\v C}ech nerve of a good cover
  is cofibrant in the projective model structure
  $[C^{\mathrm{op}}, \mathrm{sSet}]_{\mathrm{proj}}$
  and hence also in its localization
  $[C^{\mathrm{op}}, \mathrm{sSet}]_{\mathrm{proj}, \mathrm{loc}}$.
\end{remark}
In order to discuss descent over $C$ it is convenient to 
introduce the following notation for `cohomology over the site $C$'.
For the moment this is just an auxiliary technical notion. Later
we will see how it relates to an intrinsically defined notion of 
cohomology.
\begin{definition}
 For $C$ an  $\infty$-cohesive site,
 $A \in [C^{\mathrm{op}}, \mathrm{Set}]_{\mathrm{proj}}$ 
 fibrant, and $\{U_i \to U\}$ a good cover in $U$, we write 
 $$
   H^0_C(\{U_i\},A) := \pi_0 \mathrm{Maps}(\check{C}(\{U_i\}), A)
   \,.
 $$
 Moreover, if $A$ is equipped with the structure of a group object (respectively an 
 abelian group object)
 we write
 $$
   H^n_C(\{U_i\},A) := \pi_0 \mathrm{Maps}(\check{C}(\{U_i\}), \Wbar^n A),   
 $$
 if $n=1$ (respectively $n\geq 1$).  
 Here $\mathrm{Maps}(-,-)$ denotes the usual simplicial mapping 
 space in $[C^{\mathrm{op}},\mathrm{sSet}]$.  
\end{definition}

As is described in \cite{Jardine-Fields} the homotopy groups of a 
simplicial set $X$ have a base-point free interpretation 
as group objects over $X_0$: one defines $\pi_0(X)$ as a colimit in the usual  
way as 
\[
\pi_0(X) = \varinjlim (X_1\rightrightarrows X_0), 
\]
and, for any integer $n\geq 1$, one defines 
\[
\pi_n(X) = \bigsqcup_{x\in X_0} \pi_n(X,x),  
\]
so that $\pi_n(X)\to X_0$ has a natural structue as 
a group object over $X_0$.  If now $X\in [C^{\mathrm{op}},\mathrm{sSet}]$ 
we can perform these constructions object-wise to form presheaves 
$\pi_0^{\mathrm{PSh}}(X)$ and $\pi_n^{\mathrm{PSh}}(X)$, so that 
\[
\pi_n^{\mathrm{PSh}}(X)(U) = \bigsqcup_{x_U\in X_0(U)} \pi_n(X(U),x_U) 
\]
for instance.  Note that both constructions are functorial in $X$, and 
that $\pi_n^{\mathrm{PSh}}(X)\to X_0$ is 
a group object over $X_0$ in $[C^{\mathrm{op}},\mathrm{sSet}]$.  
If $x_U\in X_0(U)$ then we define the presheaf $\pi_n(X,x_U)$ by the 
pullback diagram 
\[
\xymatrix{ 
\pi_n(X,x_U) \ar[d] \ar[r] & \pi_n(X) \ar[d] \\ 
U \ar[r]^{x_U} & X_0 }
\]
so that $\pi_n(X,x_U)$ is naturally a presheaf of groups on the slice $C_{/U}$.  
Following \cite{Jardine-Fields} we make the following definition.     
\begin{definition}
 \label{HomotopySheaves} 
Let $C$ be a site, and let 
\[
\pi_0^{\mathrm{PSh}}\colon [C^{\mathrm{op}},\mathrm{sSet}]\to 
 \mathrm{PSh}(C)
\]
and 
\[
\pi_n^{\mathrm{PSh}}\colon [C^{\mathrm{op}},\mathrm{sSet}] 
\to \mathrm{PSh}(C),
\]
for $n\geq 1$, denote the functors described above.  We write 
\[
\pi_0\colon [C^{\mathrm{op}},\mathrm{sSet}]\to 
\mathrm{Sh}(C)
\]
and, for $n\geq 1$, 
\[
\pi_n\colon [C^{\mathrm{op}},\mathrm{sSet}]\to \mathrm{Sh}(C) 
\]
for their sheafified versions.  
\end{definition} 
Note that if $X$ is a simplicial presheaf on $C$, then $\pi_n(X)$ 
is naturally a group object over the sheaf associated to $X_0$.  
Using this we can state the main definition of this section. 
\begin{definition}
  \label{CAcyclic}
  An object $A \in [C^{\mathrm{op}}, \mathrm{sSet}]$
  is called \emph{$C$-acyclic} if 
  \begin{enumerate}
    \item it is fibrant in $[C^{\mathrm{op}}, \mathrm{sSet}]_{\mathrm{proj}}$;
    \item for all $n \in \mathbb{N}$ we have $\pi_n^{\mathrm{PSh}}(A) = \pi_n(A)$, in other words 
  the homotopy group presheaves 
  from Definition~\ref{HomotopySheaves}
  are already sheaves;
    \item the sheaves $\pi^{\PSh}_n(A)$ are acyclic with respect to 
    good covers; i.e.\ for every object $U$, for every point $a_U\in A_0(U)$, and for all 
    good covers $\{U_i\to U\}$ of $U$, we have 
    \[
      H^{1}_C(\{U_i\},\pi_1(A,a_U)) = 1
    \]
    and 
    \[
    H^k_C(\{U_i\},\pi_n(A,a_U)) = 1
    \]
    for all $k \geq 1$ if $n\geq 2$. 
  \end{enumerate}
\end{definition}
\begin{remark}
  This definition can be formulated and the following statements about it
  are true over any site whatsoever. 
  However, on generic sites $C$ the $C$-acyclic objects are not 
  very interesting. They become interesting on sites such as the $\infty$-cohesive
  sites considered here, whose topology sees all their objects as being contractible.
\end{remark}
\begin{observation}
  \label{LoopingOfHCAcyclicObjects}
  If $A$ is $C$-acyclic then $\Omega_x A$ is $C$-acyclic for every point $x : * \to A$ 
  (for any model of the loop space 
  object in $[C^{\mathrm{op}}, \mathrm{sSet}]_{\mathrm{proj}}$).
\end{observation}
\proof
   The standard statement in $\mathrm{sSet}_{\mathrm{Quillen}}$
   $$
     \pi_n \Omega X \simeq \pi_{n+1} X
   $$
   directly prolongs 
   to $[C^{\mathrm{op}}, \mathrm{sSet}]_{\mathrm{proj}}$.
\endofproof
\begin{theorem}
  \label{CAcyclicityAndLocalFibrancy}
  Let $C$ be an $\infty$-cohesive site.
  Sufficient conditions for an object 
  $A \in [C^{\mathrm{op}}, \mathrm{sSet}]$ 
  to be fibrant in the local 
  model structure 
  $[C^{\mathrm{op}}, \mathrm{sSet}]_{\mathrm{proj}, \mathrm{loc}}$
  are
  \begin{itemize}
    \item 
      $A$ is 0-$C$-truncated and $C$-acyclic;
    \item 
      $A$ is $C$-connected and $C$-acyclic;
    \item
      $A$ is a group object and $C$-acyclic.
  \end{itemize}
\end{theorem}
Here and in the following ``$C$-truncated'' and 
``$C$-connected'' means: as simplicial presheaves
(not after sheafification of homotopy presheaves).  So for example, here and in the 
following a simplicial presheaf $X$ is $C$-connected if it takes values in connected
simplicial sets.
\begin{remark}
  This means that with $A$ satisfying the conditions 
  of Theorem~\ref{CAcyclicityAndLocalFibrancy} above, with $X$ any simplicial presheaf
  and $\xymatrix{Y \ar[r]^{\simeq} & X}$ a {\em split} hypercover 
  (see Definition 4.8 of \cite{DHI}), the 
  cocycle $\infty$-groupoid $\mathbf{H}(X, A)$ is presented by 
  simplicial function complex
  $[C^{\mathrm{op}}, \mathrm{sSet}](Y, A)$.
  The vertices of this simplicial set are 
  simplicial hyper-{\v C}ech cocycles with coefficients in $A$,
  the edges are {\v C}ech coboundaries and so on. Specifically, if $\{U_i \to X\}$
  is a \emph{good cover} in that all finite non-empty intersections of patches
  are representable, then the {\v C}ech nerve $\check{C}(\{U_i\}) \to X$ is a split 
  hypercover, and a morphism of simplicial presheaves $\check{C}(\{U_i\}) \to A$
  is a hyper-{\v C}ech cocycle with respect to the given cover.
\end{remark}
We demonstrate Theorem \ref{CAcyclicityAndLocalFibrancy} 
in several stages in the following list of propositions.
\begin{lemma}
  \label{DescentOf0TruncatedObjects}
  A 0-$C$-truncated object is fibrant in
  $[C^{\mathrm{op}}, \mathrm{sSet}]_{\mathrm{proj}, \mathrm{loc}}$
  precisely if it is fibrant in 
  $[C^{\mathrm{op}}, \mathrm{sSet}]_{\mathrm{proj}}$
  and weakly equivalent to a sheaf: to an object in the image of the 
  canonical inclusion
  $$
   \mathrm{Sh}(C) \hookrightarrow
  [C^{\mathrm{op}}, \mathrm{Set}] \hookrightarrow
  [C^{\mathrm{op}}, \mathrm{sSet}]
    \,.
  $$
\end{lemma}
\proof
  From general facts of left Bousfield localization we have 
  that the fibrant objects in the local model structure
  are necessarily fibrant also in the global structure.
  Since moreover $A \to \pi_0(A)$ is a weak equivalence in the 
  global model structure by assumption, we have for every
  covering $\{U_i \to U\}$ in $C$
  a sequence of weak equivalences
  \begin{align*}
    \mathrm{Maps}(\check{C}(\{U_i\}), A) & \simeq 
    \mathrm{Maps}( \check{C}(\{U_i\}), \pi_0(A))  \\   
     & \simeq
    \mathrm{Maps}( \pi_0 \check{C}(\{U_i\}), \pi_0(A))    \\
     & \simeq
    \mathrm{Sh}_C(S(\{U_i\}), \pi_0(A)),    
  \end{align*}
  where $S(\{U_i\}) \hookrightarrow U$ is the sieve corresponding to the
  cover. Therefore the descent condition
  $$
    \mathrm{Maps}(U, A)\stackrel{\simeq}{\to}
    \mathrm{Maps}(\check{C}(\{U_i\}), A)
  $$
  is precisely the sheaf condition for $\pi_0(A)$.
\endofproof
\begin{lemma}
  \label{DescentOnPi0AndLooping}
  A pointed and $C$-connected fibrant object $A \in [C^{\mathrm{op}}, \mathrm{sSet}]_{\mathrm{proj}}$ 
  is fibrant in 
  $[C^{\mathrm{op}}, \mathrm{sSet}]_{\mathrm{proj}, \mathrm{loc}}$ 
  if for all objects $U \in C$
  \begin{enumerate}
    \item 
      $H^0_C(U, A) \simeq *$; 
    \item
      $\Omega_* A$ is fibrant in 
        $[C^{\mathrm{op}}, \mathrm{sSet}]_{\mathrm{proj}, \mathrm{loc}}$\,,
  \end{enumerate}
  where $\Omega_* A$ is any fibrant object in 
  $[C^{\mathrm{op}}, \mathrm{sSet}]_{\mathrm{proj}}$ representing the
  simplicial looping of $A$.
\end{lemma}
\proof
  For $\{U_i \to U\}$ a  good covering of an object $U$ we need to show that the 
  canonical morphism
  $$
    \mathrm{Maps}(U, A)
    \to
    \mathrm{Maps}(\check{C}(\{U_i\}), A)
  $$
  is a weak homotopy equivalence. This is equivalent to 
  the two morphisms
  \begin{enumerate}
    \item 
      $\pi_0 \mathrm{Maps}(U, A)
      \to
      \pi_0\mathrm{Maps}(\check{C}(\{U_i\}), A)$
    \item
       $\Omega_* \mathrm{Maps}(U, A)
       \to
      \Omega_* \mathrm{Maps}(\check{C}(\{U_i\}), A)$
  \end{enumerate}
  being weak equivalences.
  Since $A$ is $C$-connected the first of these says that there is 
  a weak equivalence
  $* \stackrel{\simeq}{\to} H_C^0(U,A)$.
  The second condition is equivalent to 
       $\mathrm{Maps}(U, \Omega_* A)
       \to
      \mathrm{Maps}(\check{C}(\{U_i\}), \Omega_* A)$,
      being a weak equivalence, hence to the descent of $\Omega_* A$.
\endofproof
\begin{lemma}
  \label{DescentFor1Truncated}
  An object $A$ which is $C$-connected, 1-$C$-truncated and $C$-acyclic is
  fibrant in $[C^{\mathrm{op}}, \mathrm{sSet}]_{\mathrm{proj}, \mathrm{loc}}$.
\end{lemma}
\proof
  The first condition of Lemma \ref{DescentOnPi0AndLooping}
  holds by the third condition of $C$-acyclicity.
  The second condition in Lemma \ref{DescentOnPi0AndLooping}
  is that $\pi_1(A)$ satisfies descent.
  By $C$-acyclicity this is a sheaf and it is 0-truncated
  by assumption, therefore it  
  satisfies descent by Lemma \ref{DescentOf0TruncatedObjects}.
\endofproof
\begin{proposition}
  \label{DescentOfConnectedHCacyclicObjects}
  Every pointed $C$-connected and $C$-acyclic object 
  $A \in [C^{\mathrm{op}}, \mathrm{sSet}]_{\mathrm{proj}}$ 
 is fibrant in 
  $[C^{\mathrm{op}}, \mathrm{sSet}]_{\mathrm{proj}, \mathrm{loc}}$.
\end{proposition}
\proof
  We first show the statement for truncated $A$
  and afterwards for the general case.
  The $k$-truncated case in turn we consider by 
  induction over $k$.
  If $A$ is 1-truncated the proposition holds by
  Lemma~\ref{DescentFor1Truncated}.
  Assuming then that the statement has been shown for 
  $k$-truncated $A$, we need to show it for $(k+1)$-truncated $A$.

  We achieve this by decomposing $A$ into its Moore-Postnikov tower
  $$
    A \to \cdots \to A(n+1) \to A(n) \to \cdots \to *\,.
  $$
  It is a standard fact 
  (shown in \cite{GoerssJardine}, VI Theorem 3.5 for simplicial 
  sets, which generalizes immediately to the global model structure
  $[C^{\mathrm{op}}, \mathrm{sSet}]_{\mathrm{proj}}$ ) that 
  for all $n > 1$ we have sequences
  $$
    K(n) \to A(n) \to A(n-1)
    \,,
  $$
  where
  $A(n-1)$ is $(n-1)$-truncated with homotopy groups in degree
  $\leq n-1$ those of $A$, and 
  where the right morphism is a Kan fibration and the left 
  morphism is its kernel, such that 
  $$
    A = \varprojlim A(n)
    \,.
  $$
  Moreover, there are canonical weak homotopy equivalences
  $$
    K(n) \to \Xi ((\pi_{n-1} A)[n])
  $$
  to the Eilenberg-MacLane object on the $(n-1)$-st homotopy group in degree $n$.
  Since $A(n-1)$ is $(n-1)$-truncated and connected, the 
  induction assumption implies that it is fibrant in the 
  local model structure.
  
  Moreover we see that $K(n)$ is fibrant in 
  $[C^{\mathrm{op}}, \mathrm{sSet}]_{\mathrm{proj}, \mathrm{loc}}$:
  the first condition of \ref{DescentOnPi0AndLooping}
  holds by the assumption that $A$ is $C$-connected. The
  second condition is implied again by the induction hypothesis, 
  since $\Omega K(n)$ is $(n-1)$-truncated, connected and 
  still $C$-acyclic, by Observation \ref{LoopingOfHCAcyclicObjects}.

  Therefore in the diagram (where $\mathrm{Maps}(-,-)$ denotes the 
  simplicial hom complex)
 $$
   \xymatrix{
     \mathrm{Maps}(U,K(n))
     \ar[r]
     \ar[d]^\simeq
     &
     \mathrm{Maps}(U,A(n))
     \ar[r]
     \ar[d]
     &
     \mathrm{Maps}(U,A(n-1))
     \ar[d]^\simeq
     \\
     \mathrm{Maps}(\check{C}(\{U_i\}),K(n))
     \ar[r]
     &
     \mathrm{Maps}(\check{C}(\{U_i\}),A(n))
     \ar[r]
     &
     \mathrm{Maps}(\check{C}(\{U_i\}),A(n-1))     
   }
 $$
 for $\{U_i \to U\}$ any good cover in $C$
 the top and bottom rows are fiber sequences 
 (notice that all simplicial sets in the top row are connected because $A$ is connected) 
 and the left and
 right vertical morphisms are weak equivalences in 
 $[C^{\mathrm{op}}, \mathrm{sSet}]_{\mathrm{proj}}$
 (the right one since $A(n-1)$ is fibrant in the local model structure
 by induction hypothesis, as remarked before,
 and the left one by $C$-acyclicity of $A$). It follows
 that also the middle morphism is a weak equivalence.
 This shows that $A(n)$ is fibrant in 
 $[C^{\mathrm{op}}, \mathrm{sSet}]_{\mathrm{proj}, \mathrm{loc}}$.
 By completing the induction the same then follows for the 
 object $A$ itself.
 
 This establishes the claim for truncated $A$. To demonstrate
 the claim for general $A$ notice that the limit over a sequence
 of fibrations between fibrant objects 
 is a homotopy limit. 
 Therefore we have
 $$
   \raisebox{20pt}{
   \xymatrix{
      \mathrm{Maps}(U, A)
      \ar[d]
      & \simeq &
      \varprojlim_n\mathrm{Maps}(U, A(n))
      \ar[d]^{\simeq}
      \\
      \mathrm{Maps}(\check{C}(\{U_i\}), A)
      & \simeq &
      \varprojlim_n\mathrm{Maps}(\check{C}(\{U_i\}), A(n))
   }
   }
 $$
 where the right vertical morphism is a morphism between homotopy
 limits in $[C^{\mathrm{op}}, \mathrm{sSet}]_{\mathrm{proj}}$
 induced by a weak equivalence of diagrams, hence is itself a weak equivalence.
 Therefore $A$ is fibrant in 
 $[C^{\mathrm{op}}, \mathrm{sSet}]_{\mathrm{proj}, \mathrm{loc}}$.
\endofproof
\begin{lemma}
  \label{CanonicalFiberSequenceForSimplicialGroup}
  For $G \in [C^{\mathrm{op}}, \mathrm{sSet}]$ 
  a group object, the canonical sequence
  $$
    G_0 \to G \to G/G_0
  $$
  is a homotopy fiber sequence in 
  $[C^{\mathrm{op}}, \mathrm{sSet}]_{\mathrm{proj}}$.
\end{lemma}
\proof
  Since homotopy pullbacks of presheaves are computed objectwise,
  it is sufficient to show this for $C = *$, hence in
  $\mathrm{sSet}_{\mathrm{Quillen}}$.
  One checks that generally, for $X$ a Kan complex and $G$ a simplicial group
  acting on $X$, the quotient morphism $X \to X/G$ is a Kan fibration.
  Therefore the homotopy fiber of $G \to G/G_0$ is presented by the
  ordinary fiber in $\mathrm{sSet}$. 
  Since the action of $G_0$ on $G$ is free, this is indeed $G_0 \to G$.   
\endofproof
\begin{proposition}
  Every $C$-acyclic group object 
  $G \in [C^{\mathrm{op}}, \mathrm{sSet}]_{\mathrm{proj}}$ 
  for which $G_0$ is a sheaf
  is fibrant in 
  $[C^{\mathrm{op}}, \mathrm{sSet}]_{\mathrm{proj}, \mathrm{loc}}$. 
\end{proposition}
\proof
  By lemma \ref{CanonicalFiberSequenceForSimplicialGroup}
  we have a fibration sequence
  $$  
     G_0 \to G \to G/G_0
     \,.
  $$
  Since $G_0$ is assumed to be a sheaf it is fibrant in the local
  model structure by 
  Lemma~\ref{DescentOf0TruncatedObjects}.
  Since $G/G_0$ is evidently connected and $C$-acyclic 
  it is fibrant in the local model structure by 
  Proposition~\ref{DescentOfConnectedHCacyclicObjects}.
  As before in the proof there this implies that also
  $G$ is fibrant in the local model structure.
\endofproof
This completes the proof of Theorem \ref{CAcyclicityAndLocalFibrancy}.

\subsubsection{By cocycles in a category of fibrant objects}
\label{Categories of fibrant objects}

We discuss here a presentation  
of the hom-$\infty$-groupoids of an
$\infty$-category which itself is presented by the homotopical structure
known as a \emph{category of fibrant objects} \cite{Brown}. The resulting
presentation is much `smaller' than the 
general Dwyer-Kan simplicial localization \cite{DwyerKanComputations}: 
where the latter
encodes a morphism in the localization by a zig-zag of arbitrary length 
(of morphisms in the presenting category), the following 
Theorem \ref{SimplicialLocalizationOfCatOfFibrantObjects} asserts that 
with the structure of a category of fibrant objects, we may restrict to zig-zags
of length 1.
A slight variant of this statement has been proven by Cisinski in \cite{Cisinski}. 
The following subsumes this variant and provides a maybe more direct proof.
\medskip

Before describing the hom-$\infty$-groupoids, 
we briefly recall some basic notions and facts from \cite{Brown}.
\begin{definition}[Brown]
  \label{CategoryOfFibrantObjects}
A \emph{category of fibrant objects} is a category $\mathcal{C}$ with finite products, which comes 
equipped with two distinguished full subcategories $\mathcal{C}_F$ and 
$\mathcal{C}_W$, whose morphisms are called \emph{fibrations} and \emph{weak equivalences} 
respectively, such that the following properties hold: 
\begin{enumerate}

\item $\mathcal{C}_F$ and $\mathcal{C}_W$ contain all of the isomorphisms 
of $\mathcal{C}$, 

\item weak equivalences satisfy the 2-out-of-3 property,

\item the subcategories $\mathcal{C}_F$ and $\mathcal{C}_F\cap \mathcal{C}_W$ 
are stable under pullback, 

\item there exist functorial path objects in $\mathcal{C}$.
\end{enumerate}
Morphisms in $\mathcal{C}_F\cap \mathcal{C}_W$ are called {\em acyclic 
fibrations}.
\end{definition}
The axioms for a category of fibrant objects give roughly 
half of the structure of a model category, however these axioms 
still suffice to give a calculus-of-fractions description of the associated 
homotopy category.  

\begin{examples}
  \label{BasicExamplesOfCatsOfFibObjects}
We have the following well known examples of categories of fibrant objects.  
\begin{itemize}
\item For any model category (with functorial factorization) the full subcategory of fibrant objects is a category of fibrant objects. 

\item The category of stalkwise Kan simplicial 
presheaves on any site with enough points. 
In this case the fibrations are the stalkwise fibrations 
and the weak equivalences are the stalkwise weak equivalences.
\end{itemize}
\end{examples}
\begin{remark}
 \label{StalkwiseFibrationsAreNotModelStructureFibrations}
Notice that (over a non-trivial site) the second example above is \emph{not} 
a special case of the first:
while there are model structures on categories of simplicial presheaves whose
weak equivalences are the stalkwise weak equivalences, their fibrations 
(even between fibrant objects) are much more 
restricted than just being stalkwise fibrations.  
\end{remark}

We will use repeatedly the following consequence of the 
axioms of a category of fibrant objects (this is called the 
{\em cogluing lemma} in \cite{GoerssJardine} where it appears 
as Lemma 8.10, Chapter II).  
\begin{lemma} 
\label{cogluing lemma} 
Let $\mathcal{C}$ be a category of fibrant objects.  Suppose given 
a diagram 
\[
\xymatrix{ 
A_1 \ar[d]_-{f_A} \ar[r]^-{p_1} & B_1 \ar[d]_-{f_B} & 
C_1 \ar[d]_-{f_C} \ar[l] \\ 
A_2 \ar[r]_-{p_2} & B_2 & C_2 \ar[l] } 
\]
in which $p_1$ and $p_2$ are fibrations, and $f_A$, $f_B$ and $f_C$ are 
weak equivalences.  Then the induced map 
\[
A_1\times_{B_1}C_1 \to A_2\times_{B_2} C_2 
\]
is also a weak equivalence.   
\end{lemma}

We now come to the discussion of the hom-$\infty$-groupoids presented by
$\mathcal{C}$
\begin{definition}
Let $\mathcal{C}$ be a category 
of fibrant objects and let $X$ and $A$ be objects of $\mathcal{C}$. 

Write $\mathrm{Cocycle}(X,A)$ for the category whose
\begin{itemize}
  \item objects are spans, hence diagrams in $\mathcal{C}$ of the form
   $$
     \xymatrix{
	   X \ar@{<<-}[r]^p_{\simeq} & Y \ar[r]^g & A
	 }
	 \,,
   $$
   such that the left morphism is an acyclic fibration;
  \item
    morphisms $f : (p_1,g_1) \to (p_2, g_2)$ are given by 
	morphisms $f : X \to Y$ in $\mathcal{C}$, making the diagram
	$$
		\begin{xy} 
		(0,7.5)*+{Y_1}="1"; 
		(-15,0)*+{X}="2"; 
		(15,0)*+{A}="3"; 
		(0,-7.5)*+{Y_2}="4"; 
			{\ar@{->>}_-{\simeq} "1";"2"}; 
			{\ar "1";"4"}; 
			{\ar^{g_1} "1";"3"};
			{\ar_{g_2} "4";"3"};
			{\ar@{->>}^-{\simeq} "4";"2"};
	\end{xy}
   $$
   commute.
\end{itemize}
Similarly write $\mathrm{wCocycle}(X,A)$ for the category defined
analogously, where however the left legs are only required to be weak 
equivalences, not necessarily fibrations.
\end{definition}
\begin{remark}
In Section 3.3 of \cite{Cisinski} the category $\mathrm{Cocycle}(X,A)$
is denoted
$\underline{\mathrm{Hom}}_{\, \mathcal{C}}(X,A)$. In Section 1 of \cite{JardineCocycles}
the category $\mathrm{wCocycle}(X,A)$ (under different assumptions on $\mathcal{C}$)
is denoted $H(X,A)$ (and there only the connected components are analyzed).
\end{remark}
\begin{remark}
The morphisms $f$ in $\mathrm{Cocycle}(X,A)$ 
and $\mathrm{wCocycle}(X,A)$ are 
necessarily weak equivalences by the 2-out-of-3 property.  
The evident composition of spans under 
fiber product in $\mathcal{C}$ induces a 
functor 
\[
 \mathrm{Cocycle}(X,A) 
   \times 
 \mathrm{Cocycle}(A,B)
  \to 
 \mathrm{Cocycle}(X,B)
 \,,
\]
which defines the structure of a bicategory whose objects 
are the objects of $\mathcal{C}$.  
\end{remark}
\begin{theorem}
 \label{SimplicialLocalizationOfCatOfFibrantObjects}
Let $\mathcal{C}$ be a category of fibrant objects. Then 
for all objects $X,A \in \mathcal{C}$
the canonical inclusions 
\[
N\mathrm{Cocycle}(X,A) 
\to N\mathrm{wCocycle}(X,A) \to L^H\mathcal{C}(X,A)
\]
of the simplicial nerves of the categories of cocycles into 
the hom-space $L^H(X,A)$ of the \emph{hammock localization} \cite{DwyerKanComputations} 
of $\mathcal{C}$
are weak homotopy equivalences. 
\end{theorem}

As remarked above, a variant of this statement has been proven by Cisinski \cite{Cisinski} --- more precisely 
he has shown that the inclusion $N\mathrm{Cocycle}(X,A) \to L^H\mathcal{C}(X,A)$ is a weak homotopy equivalence 
(see Proposition 3.23 of \cite{Cisinski}]).  We give here a direct proof of this result, which also establishes that 
$N\mathrm{Cocycle}(X,A) \to N\mathrm{wCocycle}(X,A)$  is also a weak equivalence. 

In order to write out the proof,
	we first need a little bit of notation. By $W^{-1}\mathcal{C}^i(A,B)$ we shall mean the category which has as objects zig-zags of the form
\begin{equation*}
  \xymatrix{
     A & X_1 \ar_\sim[l]\ar[r] & X_2 \ar[r] & ... \ar[r] & X_i \ar[r] &B
  }  
  \,,
\end{equation*}
where the morphism to the left is a weak equivalence, and as morphisms ladders of weak equivalences. 
By $W^{-1}W^i(A,B)$ we denote the full subcategory where also the arrows going to the right are weak equivalences. 
Analogously we have similar categories 
$W^{-1}\mathcal{C}^i W^{-1} \mathcal{C}^j(A,B)$ and $W^{-1}W^i W^{-1} W^j(A,B)$ for pair of integers $i, j > 0$. There are obvious functors
\begin{align*}
&W^{-1}\mathcal{C}^{i+j}(A,B) \to W^{-1}\mathcal{C}^i W^{-1} \mathcal{C}^j(A,B)  \qquad \text{and}\\
&W^{-1}W^{i+j}(A,B) \to W^{-1}W^i W^{-1} W^j(A,B).
\end{align*}  
given by filling in identity morphisms. If these inclusions induce weak equivalences on nerves, 
then $\mathcal{C}$ is said to admit a 
\emph{homotopy calculus of left fractions}, see 
\cite[Section 6]{DwyerKanComputations}. In this case they show that the canonical morphism
\begin{equation*}
N\mathrm{wCocycle}(A,B) = N\big(W^{-1}\mathcal{C}(A,B)\big) \to L^H(A,B)
\end{equation*} 
is a weak homotopy equivalence 
\cite[Proposition 6.2]{DwyerKanComputations}. 
Therefore we want to show that each category of 
fibrant objects $\mathcal{C}$ admits a homotopy calculus of left fractions.

Let $F^{-1}\mathcal{C}^{i}(A,B)$ be the full 
subcategory of $W^{-1}\mathcal{C}^{i}(A,B)$ 
consisting of the zig-zags where the left going 
morphim is as acyclic fibration rather than a weak 
equivalence. Analogously we write $F^{-1}\mathcal{C}^{i}F^{-1}\mathcal{C}^{j}(A,B)$. 
Note that in either case the morphisms of these spans 
still consist of weak equivalences and not necessarily of acyclic fibrations.

\begin{lemma}\label{faclemma}
Let $\xymatrix{A & X \ar_\sim[l] \ar[r] & B}$ be a 
span in $\mathcal{C}$ where the left leg is a weak 
equivalence. Then we can find $Y \in \mathcal{C}$ and a commuting diagram 
\begin{equation*}
\xymatrix@R=0.7pc {
& & X \ar_\sim[lld] \ar[drr]\ar^\sim[dd] & &\\
A & & & & B \\
& & Y \ar@{->>}^\sim[llu] \ar[urr] & & }.
\end{equation*}
In other words we find another span 
$\xymatrix{A & Y \ar@{->>}_\sim[l] \ar[r] & B}$ 
where the left leg is an acyclic fibration and a 
morphism of spans between them. Moreover this 
assignment is functorial in the original span.
\end{lemma}
\begin{proof}
We first note that in a category of fibrant objects we can always factor a morphism $X \to Z$ as 
\begin{equation*}
\xymatrix{X \ar[r]^s & Y \ar@{->>}^p[r] & Z}.
\end{equation*}
where $p$ is a fibration and $s$ is a weak equivalence 
which is a section of an acyclic fibration $\hat Z \tfright X$. 
To see this set $Y := X \times _Z \times Z^I$. 

Now a span between $A$ and $B$ is the same as a 
morphism $X \to A \times B$. Applying the factorization to this morphism yields a diagram 
\begin{equation*}
X \to Y \to A \times B
\end{equation*}
which tranlates exactly into the diagram from above. 
It only remains to check that the left leg is indeed a 
fibration since it is clearly a weak equivalence by the 
2-out-of-3 property. This follows by the fact that it 
can be expressed as the composition $Y \to A \times 
B \xrightarrow p A$ where $p$ is the projection to the first factor which is a fibration since $B$ is fibrant.
\end{proof}

\begin{lemma}\label{lemmasix}
The following inclusion functors induce weak equivalences on nerves for all $i,j > 0$.
\begin{align*}
F^{-1}\mathcal{C}^{i}(A,B) & \longrightarrow  W^{-1}\mathcal{C}^{i}(A,B)&  \\
F^{-1}W^{i}(A,B) & \longrightarrow  W^{-1}W^{i}(A,B)&  \\
F^{-1}\mathcal{C}^{i}F^{-1}\mathcal{C}^{j}(A,B) & \longrightarrow  W^{-1}\mathcal{C}^{i}W^{-1}\mathcal{C}^{j}(A,B)& \\
F^{-1}W^{i}F^{-1}W^{j}(A,B) & \longrightarrow  W^{-1}W^{i}W^{-1}W^{j}(A,B)& \\
F^{-1}\mathcal{C}^{i+j}(A,B) &\longrightarrow   F^{-1}\mathcal{C}^{i}F^{-1}\mathcal{C}^{j}(A,B)& \\
F^{-1}W^{i+j}(A,B) &\longrightarrow   F^{-1}W^{i}F^{-1}W^{j}(A,B).&
\end{align*}
\end{lemma}
\begin{proof}
We explicitly construct functors which are homotopy inverses on nerves. For the first functor 
$K_1: F^{-1}\mathcal{C}^{i}(A,B) \longrightarrow 
 W^{-1}\mathcal{C}^{i}(A,B)$ we define an inverse 
 $L_1: W^{-1}\mathcal{C}^{i}(A,B) \longrightarrow 
 F^{-1}\mathcal{C}^{i}(A,B)$ using the factorization from the last lemma as
\begin{eqnarray*}
&\Big(\xymatrix{A & X_1 \ar@{->}_\sim[l]\ar[r] & X_2 \ar[r] & ... \ar[r] & X_i \ar[r] &B}\Big)   \\
 &\qquad \qquad \qquad\mapsto \qquad \Big(\xymatrix{A & Y_1 \ar@{->>}_\sim[l]\ar[r] & X_2 \ar[r] & ... \ar[r] & X_i \ar[r] &B} \Big) .
\end{eqnarray*}
One checks that this indeed forms a functor and that the morphism
$X_1 \xrightarrow\sim Y_1$ from the last lemma form natural transformations $id \Rightarrow K_1 \circ G_1 $ and
$id \Rightarrow G_1 \circ K_1$. This shows that on nerves $NK_1$ and $NL_1$ are homotopy inverses.

Now the functor $L_1$ restricts to a functor 
$W^{-1}W^{i}(A,B) \to F^{-1}W^{i}(A,B)$ which is homotopy inverse to the second functor of the lemma.

For the third funtor $K_2: F^{-1}\mathcal{C}^{i}F^{-1}
\mathcal{C}^{j}(A,B) \to W^{-1}\mathcal{C}^{i}W^{-1}
\mathcal{C}^{j}(A,B)$ in the lemma an inverse $L_2:  
W^{-1}\mathcal{C}^{i}W^{-1}\mathcal{C}^{j}(A,B) 
\longrightarrow F^{-1}\mathcal{C}^{i}F^{-1}\mathcal{C}^{j}(A,B)$ is similarily constructed as
\begin{eqnarray*}
&\Big(\xymatrix@C=1.3pc{A & X_1 \ar@{->}_\sim[l]\ar[r]& 
X_2 \ar[r] &  &  \ar[r] & X_i  & X_{i+1} \ar[r] \ar@{->}_\sim[l] & X_{i+2}\ar[r] & & \ar[r] &B}\Big)   \\
 &\qquad \mapsto \qquad \Big(\xymatrix@C=1.3pc{A & Y_1 
 \ar@{->>}_\sim[l]\ar[r]& X_2 \ar[r] & & \ar[r] & X_i  & Y_{i+1} \ar[r] \ar@{->>}_\sim[l] & X_{i+2}\ar[r] & & \ar[r] &B}\Big) .
\end{eqnarray*}
using again the factorization of Lemma \ref{faclemma}. 
The morphisms $X_1 \xrightarrow\sim Y_1$ and 
$X_{i+1} \xrightarrow\sim Y_{i+1}$ 
provide natural transformations $id \Rightarrow 
K_2 \circ L_2$ and $id \Rightarrow L_2 \circ K_2$. As before the functor $L_2$ restrict to an inverse for the fourth functor in the lemma.

Now we come to the functor 
$K_3: F^{-1}\mathcal{C}^{i+j}(A,B) \longrightarrow F^{-1}\mathcal{C}^{i}F^{-1}\mathcal{C}^{j}(A,B).$
Its homotopy inverse $L_3$ is constructed by iterated pullbacks as indicated in the following diagram
\begin{equation*}
\xymatrix@C=1.3pc@R=1.3pc{
& & X_1'\ar@{-->>}_\sim[ld]\ar@{-->}[r] & X_2'\ar@{-->}[r] 
& & \ar@{-->}[r]& X'_{i-1} \ar@{-->}[r] & X_{i+1} \ar@{->>}_\sim[ld] \ar[r] & X_{i+2} \ar[r] & & \ar[r] & B \\
& X_1 \ar@{->>}_\sim[ld]\ar[r] & X_2 \ar[r]&  &\ar[r]& X_{i-1} \ar[r]& X_i  & & & & \\
A & & & & & & & & & & &
}
\end{equation*}
So, using this notiation, the functor $L_3: F^{-1}\mathcal{C}^{i}F^{-1}
\mathcal{C}^{j}(A,B) \to F^{-1}\mathcal{C}^{i+j}(A,B)$ can be defined as
 \begin{eqnarray*}
&\Big(\xymatrix@C=1.3pc{A & X_1 \ar@{->>}_\sim[l]\ar[r]& 
X_2 \ar[r] &  &  \ar[r] & X_{i-1}\ar[r]  & X_{i}  & X_{i+1}\ar[r]\ar@{->>}_\sim[l] & & \ar[r] &B}\Big)   \\
 &\qquad \mapsto \qquad \Big(\xymatrix@C=1.3pc{A & X_1' 
 \ar@{-->>}_\sim[l]\ar@{-->}[r]& X_2' \ar@{-->}[r] & & \ar@{-->}[r] & X_{i-1}' \ar@{-->}[r] & X_{i+1}\ar[r] & & \ar[r] &B}\Big) .
\end{eqnarray*}
The structure maps of the pullbacks $X_i' \to X_i$ and the 
map $X_{i+1} \to X_i$ provide a natural transformation 
$K_3 \circ L_3 \Rightarrow id$. The other composition 
$L_3 \circ K_3$ consists essentially of pullbacks along 
the identity and is therefor also naturally isomorphic to 
the identity. Finally we note that the functor $L_3$ restricts 
to a functor $F^{-1}W^{i}F^{-1}W^{j}(A,B) \to F^{-1}W^{i+j}(A,B)$ 
by iterated use of the 2-out-of-3 property. Hence it also provides an inverse for the last functor of the lemma.
\end{proof}

\begin{lemma}
Each category of fibrant objects $\mathcal{C}$ admits a 
homotopy calculus of fractions and the chain of inclusions 
$N\mathrm{Cocycle}(A,B) \to N\mathrm{wCocycle}(A,B) 
\to L^H\mathcal{C}(A,B)$ are all homotopy equivalences.
\end{lemma}
\begin{proof}
In order to show that $\mathcal{C}$ admits a homotopy calculus of fractions we consider the following commuting diagrams of inclusions
\begin{equation*}
\!\!\xymatrix@R=1.3pc@C=1.3pc{
F^{-1}\mathcal{C}^{i+j}(A,B) \ar[r]\ar[d] &  W^{-1}\mathcal{C}^{i+j}(A,B) \ar[d] \\
F^{-1}\mathcal{C}^{i}F^{-1}\mathcal{C}^{j}(A,B) \ar[r] &  W^{-1}\mathcal{C}^{i}W^{-1}\mathcal{C}^{j}(A,B)
} 
\xymatrix@R=1.3pc@C=1.3pc{
 F^{-1}W^{i+j}(A,B) \ar[r]\ar[d] &  W^{-1}W^{i+j}(A,B) \ar[d] \\
 F^{-1}W^{i}F^{-1}W^{j}(A,B) \ar[r] &  W^{-1}W^{i}W^{-1}W^{j}(A,B)
}
\end{equation*}
By definition of a homotopy calculus of fraction we have 
to show that the two right vertical maps induce weak 
equivalences on nerves. But this follows since we know 
from the last Lemma \ref{lemmasix} that all the other maps in the diagrams are weak equivalences.
From the fact that $\mathcal{C}$ admits a homotopy calculus 
of fractions and \cite[Proposition 6.2]{DwyerKanComputations} 
we know that the canonical map $N\mathrm{wCocycle}(A,B) \to 
L^H\mathcal{C}(A,B)$ is a weak equivalence of simplicial sets. 
The map 
$N\mathrm{Cocycle}(A,B) \to N\mathrm{wCocycle}(A,B)$ is 
just the nerve of the functor $F^{-1}\mathcal{C}(A,B) \to  
W^{-1}\mathcal{C}(A,B)$ which is a weak equivalence by the last lemma.
\end{proof}
This completes the proof of Theorem \ref{SimplicialLocalizationOfCatOfFibrantObjects}.

\subsection{Principal bundles}
\label{section.PrincipalInfinityBundle}
\label{PrincipalInfBundle}

We discuss a presentation of 
the theory of \emph{principal $\infty$-bundles}
from section 3 in \cite{NSSa}.

\subsubsection{Universal simplicial principal bundles and the Borel construction}
\label{Universal princial bundles}

By Proposition~\ref{InftyGroupsBySimplicialGroups} every $\infty$-group
in an $\infty$-topos over an $\infty$-cohesive site is
presented by a \mbox{(pre)}sheaf of simplicial groups, hence by a strict 
group object $G$ in a 1-category of simplicial (pre)sheaves. We have seen in 
Section~\ref{InfinityGroupPresentations} that, for such a presentation,
the abstract delooping $\mathbf{B}G$ is presented by $\Wbar G$. By 
Theorem 3.19 in \cite{NSSa}, the theory of
$G$-principal $\infty$-bundles is essentially that of homotopy fibers
of morphisms into $\mathbf{B}G$, and hence, for such a 
presentation, that of homotopy fibers of morphisms into $\Wbar G$. 
By Proposition~\ref{ConstructionOfHomotopyLimits} such homotopy fibers are
computed as ordinary pullbacks of fibration resolutions of the point inclusion
into $\Wbar G$. Here we discuss these fibration resolutions. They turn out 
to be the classical \emph{universal simplicial principal bundles} $W G \to \Wbar G$
of Definition \ref{WGToWbarG}.

\medskip

Let $C$ be a site; we consider group objects in 
$[C^{\mathrm{op}}, \mathrm{sSet}]$. 
In the following let $P \in [C^{\mathrm{op}}, \mathrm{sSet}]$ be an object 
equipped with an action $\rho : P \times G \to P$ by a group object $G$.
Since sheafification preserves finite limits,
all of the following statements hold verbatim also in the category 
$\mathrm{sSh}(C)$
of simplicial sheaves over $C$. 
\begin{definition}
 \label{BisimplicialActionGroupoid}
The \emph{action groupoid object}
$$
  P/\!/G \in [\Delta^{\mathrm{op}},[C^{\mathrm{op}}, \mathrm{sSet}]]
$$
is the simplicial object in $[C^\mathrm{op},\mathrm{sSet}]$ whose $n$-simplices are
\[
  (P/\!/G)_n :=  P\times G^{\times^n} \;\;\in [C^{\mathrm{op}}, \mathrm{sSet}]
  \,,
\]
whose face maps are given on elements by 
\[
d_i(p,g_1,\ldots, g_n) = \begin{cases} 
(pg_1,g_2,\ldots, g_n) & \text{if}\ i=0, \\ 
(p,g_1,\ldots, g_ig_{i+1},\ldots, g_n) & \text{if}\ 1\leq i\leq n-1, \\ 
(p,g_1,\ldots, g_{n-1}) & \text{if}\ i=n, 
\end{cases} 
\]
and whose degeneracy maps are given on elements by 
\[
s_i(p,g_1,\ldots, g_n) = (p,g_1,\ldots, g_{i-1},1,g_i,\ldots, g_n) 
\,.
\]
\end{definition}
\begin{definition}
\label{WeakQuotient}
Write 
$$
  P/_h G := \sigma_*(P/\!/G)\in [C^{\mathrm{op}}, \mathrm{sSet}]
$$
for the total simplicial object, Definition \ref{TotalSimplicialSet}.
\end{definition} 
\begin{remark}
  According to Corollary~\ref{SimplicialHomotopyColimitByCodiagonal}
  the object $P/_h G$ presents the homotopy colimit over the simplicial object
  $P/\!/G$. We say that $P/_h G$ is the \emph{homotopy quotient} of $P$ by the
  action of $G$.
\end{remark}
\begin{example}
  \label{ActionOfSimplicialGroupOnPointAndOnItself}
  The unique trivial action of a group object $G$ on the terminal object $*$
  gives rise to a canonical action groupoid $*/\!/G$. 
  According to Definition~\ref{BarWAsCompositeWithTotal} we have
  $$
    * /_h G = \Wbar G
	\,.
  $$
  The multiplication morphism $G \times G \to G$ regarded as an 
  action of $G$ on itself gives rise to a canonical action groupoid $G/\!/G$.
  The terminal morphism $G \to *$ induces a morphism of simplicial objects
  $$
    G/\!/G \to * /\!/G
	\,.
  $$
  Defined this way $g \in G$ acts naturally from the \emph{left} on
  $G/\!/G$. To adhere to our convention that actions on bundles are right actions,
  we consider instead the right action of $g \in G$ on $G$ given by
  left multiplication by $g^{-1}$.
  With respect to this action, the action groupoid object $G/\!/G$ 
  is canonically equipped with the right $G$-action
  by multiplication from the right. Whenever in the following we write
  $$
    G/\!/G \to */\!/G
  $$
  we are referring to this latter definition.
\end{example}
\begin{definition}
   \label{UniversalSimplicialPrincipalBundle}
   Given a group object in $[C^{\mathrm{op}}, \mathrm{sSet}]$,
   write $W G \to \Wbar G$ for the morphism of 
   simplicial presheaves 
   $$
    G/_h G \to */_h G
   $$
   induced on homotopy quotients, Definition~\ref{WeakQuotient}, 
   by the morphism of canonical
   action groupoid objects of example \ref{ActionOfSimplicialGroupOnPointAndOnItself}.
   
   We will call this the \emph{universal weakly $G$-principal bundle}. 
\end{definition}
\begin{remark}
  Traditionally, at least over the trivial site,  this is
known as a presentation of the \emph{universal $G$-principal simplicial bundle};
we review this traditional theory below in Section \ref{PrincipalBundlesDiscreteGeometry}.
However, when prolonged to presheaves of simplicial sets as considered
here, it is not quite accurate to speak of a genuine 
\emph{universal principal bundle}: because the pullbacks of this bundle
to hypercovers will in general only be ``weakly principal'' in a sense
that we discuss in a moment in Section \ref{Principal infinity-bundles presentations}.
Therefore it is more accurate to speak of the \emph{universal weakly $G$-principal bundle}.
\label{PrincipalityOfUniversalSimplicialBundle}
\end{remark}
   The following proposition (which appears as Lemma 10 in \cite{RobertsStevenson}) 
   justifies this terminology and the notation $WG$ (which, recall, 
   has already been used in Definition~\ref{WGToWbarG}).
\begin{proposition}
  \label{PropertiesOfUniversalGPrincipalBundle}
  For $G$ a group object in $[C^{\mathrm{op}}, \mathrm{sSet}]$, the 
  morphism $W G \to \Wbar G$ from Definition~\ref{UniversalSimplicialPrincipalBundle}
  has the following properties:
  \begin{enumerate}
	 \item it is isomorphic to the d{\'e}calage morphism $\mathrm{Dec}_0 \Wbar G \to \Wbar G$,
	  Definition~\ref{WGToWbarG},
	 \item $WG$ is canonically equipped with a right $G$-action over $\Wbar G$
	   that makes $WG\to \Wbar G$ a $G$-principal bundle.
  \end{enumerate}
\end{proposition}
In particular it follows from 2 that $WG\to \Wbar G$ is an objectwise 
Kan fibration replacement of the point inclusion $\ast\to \Wbar G$.  

   We now discuss some basic properties of the morphism $WG\to \Wbar G$.
\begin{definition}
  For $\rho : P \times G \to P$ a $G$-action in $[C^{\mathrm{op}}, \mathrm{sSet}]$, 
  we write 
  $$
    P \times_G W G := (P \times W G)/G \in [C^{\mathrm{op}}, \mathrm{sSet}]
  $$
  for the quotient by the diagonal $G$-action with respect to the
  given right $G$ action on $P$ and the canonical right $G$-action
  on $W G$ from Proposition~\ref{PropertiesOfUniversalGPrincipalBundle}. 
  We call this quotient the \emph{Borel construction} of the $G$-action 
  on $P$.
\end{definition}
\begin{proposition}
  \label{TotalSimplicialObjectByBorelConstruction}
For $P \times G \to P$ an action in $[C^{\mathrm{op}}, \mathrm{sSet}]$, there is an isomorphism
$$
  P/_h G
  =
  P\times_G WG
  \,,
$$
between the homotopy quotient, Definition~\ref{WeakQuotient},
and the Borel construction.
In particular, for all $n \in \mathbb{N}$ there are isomorphisms
\[
  (P/_h G)_n = P_n\times G_{n-1} \times \cdots \times G_0 
  \,.
\]
\end{proposition}
\begin{proof}
  This follows by a straightforward computation.
\end{proof}
\begin{lemma}
  \label{PropertiesOfSimplicialQuotientsBySimplicialGroups}
  Let $P$ be a Kan complex, $G$ a simplicial group and $\rho : P \times G \to P$
  a free action. The following holds.
  \begin{enumerate}
    \item 
	  The quotient map $P \to P/G$ is a Kan fibration.
	\item 
	  The quotient $P/G$ is a Kan complex. 
  \end{enumerate}
\end{lemma}
The second statement is for instance Lemma 3.7 in Chapter V of \cite{GoerssJardine}.
\begin{lemma}
 \label{KanSimplicialHomotopyQuotientIsKan}
 For $P$ a Kan complex and $P \times G \to P$ an action by a group object, 
 the homotopy quotient $P /_h G$, Definition~\ref{WeakQuotient}, is itself a Kan complex.
\end{lemma}
\proof
  By Proposition~\ref{TotalSimplicialObjectByBorelConstruction} the homotopy
  quotient is isomorphic to the Borel construction. Since
  $G$ acts freely on $WG$ it acts freely on $P \times WG$.
  The statement then follows with Lemma~\ref{PropertiesOfSimplicialQuotientsBySimplicialGroups}.
\endofproof
\begin{remark}
  \label{TheUniversalBundleIsTheUniversalBundle}
  \label{UniversalPrincipalInfinityBundle}
  Let $\hat X \to \Wbar G$ be a morphism in $[C^{\mathrm{op}}, \mathrm{sSet}]$,
  presenting, by Proposition~\ref{InftyGroupsBySimplicialGroups}, a morphism 
  $X \to \mathbf{B}G$
  in the $\infty$-topos $\mathbf{H} = \mathrm{Sh}_\infty(C)$.
  By  theorem 3.19 of \cite{NSSa} every $G$-principal
  $\infty$-bundle over $X$ arises as the homotopy fiber of such a morphism.  By using 
  Proposition~\ref{PropertiesOfUniversalGPrincipalBundle} together with Proposition~\ref{ConstructionOfHomotopyLimits}
  it follows that the principal $\infty$-bundle classified by $\hat X \to \Wbar G$ is 
  presented by the
  ordinary pullback of $W G \to \Wbar G$. This is the defining property of the
  universal principal bundle. 
\end{remark}
In section \ref{Principal infinity-bundles presentations} below we show how this
observation leads to a complete presentation of the theory of principal $\infty$-bundles
by {\em weakly principal simplicial bundles}.

\subsubsection{Presentation in locally fibrant simplicial sheaves}
\label{Principal infinity-bundles presentations}

We discuss a presentation of the general notion of principal $\infty$-bundles,
by \emph{weakly principal bundles}
in a 1-category of simplicial sheaves.

\medskip

Let $\mathbf{H}$ be a hypercomplete $\infty$-topos
(for instance a cohesive $\infty$-topos), which admits a 1-site $C$ with enough points. 
\begin{observation}
From  Section~\ref{InfinityToposPresentation}  
a category with weak equivalences that presents $\mathbf{H}$
under simplicial localization 
is the category $\mathrm{sSh}(C)$ of simplicial 1-sheaves 
on $C$ with the weak
equivalences $W \subset \mathrm{sSh}(C)$ 
being the stalkwise weak equivalences:
$$
  \mathbf{H} \simeq L_W \mathrm{sSh}(C)
  \,.
$$
Also the full subcategory
$$
  \mathrm{sSh}(C)_{\mathrm{lfib}} \hookrightarrow \mathrm{sSh}(C)
$$
on the locally fibrant objects is a presentation.
\label{SimplicialSheavesWithStalkwiseWeakEquivalencesModel1LocalicHypercompleteInfinityTopos}
\end{observation}
\begin{corollary}
  Regard $\mathrm{sSh}(C)_{\mathrm{lfib}}$ as a category
  of fibrant objects, Definition~\ref{CategoryOfFibrantObjects}, 
  with weak equivalences and fibrations 
  the stalkwise weak equivalences and fibrations in $\mathrm{sSet}_{\mathrm{Quillen}}$,
  respectively, as in Example \ref{BasicExamplesOfCatsOfFibObjects}.
   Then for any two objects $X, A \in \mathbf{H}$ there are simplicial sheaves, 
  to be denoted by the same symbols, such that the hom $\infty$-groupoid
  in $\mathbf{H}$ from $X$ to $A$ is presented in $\mathrm{sSet}_{\mathrm{Quillen}}$ 
  by the Kan complex of 
  cocycles from Section \ref{Categories of fibrant objects}. 
\end{corollary}
\proof
  By theorem \ref{SimplicialLocalizationOfCatOfFibrantObjects}.
\endofproof
We now discuss, for the general theory of principal $\infty$-bundles in 
$\mathbf{H}$ from \cite{NSSa}
a corresponding realization in the presentation for $\mathbf{H}$
given by $(\mathrm{sSh}(C), W)$.

By Proposition~\ref{InftyGroupsBySimplicialGroups} every $\infty$-group
in $\mathbf{H}$ is presented by an ordinary group in $\mathrm{sSh}(C)$.
It is too much to ask
that also every $G$-principal $\infty$-bundle is presented by a
principal bundle in $\mathrm{sSh}(C)$.
But something close is true: every principal $\infty$-bundle
is presented by a \emph{weakly principal} bundle in 
$\mathrm{sSh}(C)$. 

\medskip

\begin{definition}
  \label{WeaklyGPrincipalBundle}
Let $X \in \mathrm{sSh}(C)$ be any object, and let 
$G \in \mathrm{sSh}(C)$ 
be equipped with the structure of a group object. 
A {\em weakly $G$-principal bundle} is
\begin{itemize}
\item an object $P \in \mathrm{sSh}(C)$ (the {\em total space});

\item a local fibration $\pi\colon P\to X$ (the {\em bundle projection});

\item a right action 
  $$ 
    \raisebox{10pt}{
    \xymatrix{
      P \times G \ar[dr] \ar[rr]^{\rho} && P \ar[dl]
	  \\
	  & X
	}
	}
  $$ 
  of $G$ on $P$ over $X$
\end{itemize} 
such that 
\begin{itemize}
\item the action of $G$ is {\em weakly principal} in the sense that the \emph{shear map}
\[
 (p_1, \rho) :  P \times G \to P \times_X P \qquad (p,g) \mapsto (p,p g)
\]
is a local weak equivalence.
\end{itemize}
\end{definition}

\begin{remark}
We do not ask the $G$-action to be degreewise free as in \cite{JardineLuo}, 
where a similar notion is considered. However we show in 
Corollary \ref{befreiung} below that each weakly $G$-principal bundle 
is equivalent to one with free $G$-action.
\end{remark}

\begin{definition}
A morphism of weakly $G$-principal bundles $(\pi,\rho) \to (\pi',\rho')$ over $X$ 
is a morphism $f : P \to P'$ in $\mathrm{sSh}(C)$ 
that is $G$-equivariant and 
commutes with the bundle projections, hence such that it makes this diagram 
commute:
$$
  \raisebox{20pt}{
  \xymatrix{
    P \times G
	\ar[rr]^{(f, \mathrm{id})}
	\ar[d]^{\rho}
	&&
	P' \times G
	\ar[d]^{\rho'}
	\\
    P \ar[dr]_{\pi} \ar[rr]^{f} && P' \ar[dl]^{\pi'}
	\\
	& X
  }
  }
$$
Write 
$$
  \mathrm{w}G\mathrm{Bund}(X)
  \in 
  \mathrm{sSet}_{\mathrm{Quillen}}
$$
for the nerve of the category of weakly $G$-principal bundles and morphisms
as above. The $\infty$-groupoid that this presents under 
$\mathrm{Grpd}_{\infty} \simeq (\mathrm{sSet}_{\mathrm{Quillen}})^\circ$ 
(i.e. its Kan fibrant replacement), 
we call the \emph{$\infty$-groupoid of weakly $G$-principal bundles over $X$}.
\end{definition}

\begin{lemma}
 \label{multishear}
Let $\pi : P \to X$ be a weakly $G$-principal bundle.  Then the following statements are true: 
\begin{enumerate}
\item for any point $p : * \to P$ the action of $G$ induces a weak equivalence
\begin{equation*}
G \longrightarrow P_x  
\end{equation*}
where $x = \pi (p)$ and where $P_x$ is the fiber of $P\to X$ over $x$,
\item
for all $n \in \mathbb{N}$, the multi-shear maps 
\begin{equation*}
P \times G^n  \to P^{\times^{n+1}_X} \qquad (p,g_1,...,g_n) \mapsto (p,p g_1,...,p g_n)
\end{equation*}
are weak equivalences.
\end{enumerate}
\end{lemma}
\begin{proof}
We consider the first statement.  Regard the weak equivalence 
$P \times G \xrightarrow{\sim} P \times_X P$ 
as a morphism over $P$ where in both cases the map to $P$ is given by projection onto the first factor. 
By basic properties of categories of fibrant objects, both of these projections are
fibrations.
Therefore, by the cogluing lemma (Lemma~\ref{cogluing lemma}) 
the pullback of the shear map along $p$ is still a weak equivalence. 
But this pullback is just the map 
$G\to P_x$, which proves the claim.     

For the second statement, we use induction on $n$.  Suppose that 
$P\times G^n\to P^{\times^{n+1}_X}$ is a weak equivalence.  By
Lemma~\ref{cogluing lemma} again,    
the pullback 
$P^{\times^n_X}\times_X (P\times G)\to P^{\times^{n+2}_X}$ 
of the shear map $P\times G\to P\times_X  P$ along the fibration $P^{\times^n_X} \to X$ is again a 
weak equivalence.  Similarly 
the product $P\times G^n\times G\to P^{\times^{n+1}_X}\times G$
of the $n$-fold shear map with $G$ is also a weak equivalence.
The composite of these two weak equivalences is the multi-shear map
$ P \times G^{n+1} \to P^{\times^{n+2}_X}$, which is hence also a weak equivalence.
\end{proof}

\begin{proposition}
Let $P \to X$ be a weakly $G$-principal bundle and let $f: Y \to X$ be an arbitrary morphism. 
Then the pullback $f^*P \to Y$ exists and is also 
canonically a weakly $G$-principal bundle. This 
operation extends to define a pullback morphism
$$
  f^* : \mathrm{w}G\mathrm{Bund}(X) \to \mathrm{w}G\mathrm{Bund}(Y)
  \,.
$$
\end{proposition}
\begin{proof}
Again this follows by basic properties of a category of fibrant objects:
the pullback $f^*P$ exists  
and the morphism $f^*P\to Y$ is again a local fibration;  
thus it only 
remains to show that $f^*P$ is weakly principal, i.e.\  that the 
morphism $f^*P \times G \to f^*P \times_Y f^*P$ is a weak 
equivalence. This follows from Lemma~\ref{cogluing lemma} again. 
\end{proof}

\begin{remark}
The functor $f^*$ associated to the 
map $f\colon Y\to X$ above is the restriction of a functor 
$f^*\colon \mathrm{sSh}(C)/X \to \mathrm{sSh}(C)/Y $ mapping from simplicial sheaves over 
$X$ to simplicial sheaves over $Y$.  This functor $f^*$ has a left 
adjoint 
$f_!\colon \mathrm{sSh}(C)_{/Y} \to \mathrm{sSh}_{/X} $
given by composition 
along $f$, in other words 
\[
  f_!(E\to Y) = E\to Y\xrightarrow{f} X.
\]  
Note that the functor $f_!$ does not usually restrict to a functor 
$f_!\colon \mathrm{w}G\mathrm{Bund}(Y) \to \mathrm{w}G\mathrm{Bund}(X)$. 
But when it does, we say that 
principal $\infty$-bundles {\em satisfy descent along $f$}.  In this situation, 
if $P$ is a weakly $G$-principal bundle on $Y$, then $P$ is weakly equivalent 
to the pulled-back principal $\infty$-bundle $f^*f_!P$ on $Y$, in other words 
$P$ `descends' to $f_!P$.  
\end{remark}
The next result 
says that weakly $G$-principal bundles satisfy descent along 
local acyclic fibrations (hypercovers).  
 
\begin{proposition}
  \label{PushforwardOfGBundlesAlongHypercovers}
Let $p: Y \to X$ be a local acyclic fibration in $\mathrm{sSh}(C)$.  
Then the functor $p_!$ defined above restricts to 
a functor $p_!\colon \mathrm{w}G\mathrm{Bund}(Y) \to \mathrm{w}G\mathrm{Bund}(X)$, 
left adjoint to $p^*\colon \mathrm{w}G\mathrm{Bund}(X) \to 
\mathrm{w}G\mathrm{Bund}(Y)$, hence to a homotopy equivalence in $\mathrm{sSet}_{\mathrm{Quillen}}$.
\end{proposition}
\begin{proof}
Given a weakly $G$-principal bundle $P \to Y$, 
the first thing we have to check is that the map $P \times G \to P \times_X P$ 
is a weak equivalence. This map can be factored as 
$P \times G \to P\times_Y P \to P \times_X P$. 
Hence it suffices to show that the map $P \times_Y P \to P \times_X P$ 
is a weak equivalence. 
But this follows from Lemma~\ref{cogluing lemma}, 
since both pullbacks are along local fibrations and $Y \to X$ is a local weak equivalence
by assumption. 
This establishes the existence of the functor $p_!$. 
It is easy to see that it is left adjoint to $p^*$. This implies that it induces a homotopy equivalence in $\mathrm{sSet}_{\mathrm{Quillen}}$. 
\end{proof} 
\begin{corollary}
For $f: Y \to X$ a local weak equivalence,
the induced functor $f^*: \mathrm{w}G\mathrm{Bund}(X) \to \mathrm{w}G\mathrm{Bund}(Y)$ 
is a homotopy equivalence.
\end{corollary}
\begin{proof}
Using the Factorization Lemma of \cite{Brown} we can factor the weak equivalence $f$ 
into a composite of a local acyclic fibration and a right inverse to a local acyclic fibration. 
Therefore, by Proposition~\ref{PushforwardOfGBundlesAlongHypercovers}, $f^*$
may be factored as the composite of two homotopy equivalences, hence is itself a homotopy
equivalence.
\end{proof}

\medskip

We discuss now how weakly $G$-principal bundles arise from the universal
$G$-principal bundle (Definition~\ref{UniversalSimplicialPrincipalBundle}) by pullback, and how
this establishes their equivalence with $G$-cocycles.

\begin{proposition}
  For $G$ a group object in $\mathrm{sSh}(C)$,
  the map $W G \to \Wbar G$ from Definition~\ref{UniversalSimplicialPrincipalBundle}
  equipped with the $G$-action of Proposition~\ref{PropertiesOfUniversalGPrincipalBundle}
  is a weakly $G$-principal bundle.
\end{proposition}
Indeed, it is a genuine (strictly) $G$-principal bundle, in that the shear
map is an isomorphism. This is a classical fact,
for instance around Lemma 4.1 in chapter V of \cite{GoerssJardine}. In terms of
the total simplicial set functor it is  
observed in Section 4 of  \cite{RobertsStevenson}. \newline
\begin{proof}
  By inspection one finds that 
  $$
    \xymatrix{
	  (G /\!/G) \times G \ar[d]\ar[r]  & G /\!/ G \ar[d]
	  \\
	  G /\!/G \ar[r] & {*}/\!/G
	}
  $$
  is a pullback diagram in $[\Delta^{\mathrm{op}}, \mathrm{sSh}(C)]$. 
  Since the total simplicial object functor $\sigma_*$ of 
  Definition \ref{TotalSimplicialSet} 
  is right
  adjoint it preserves this pullback. This shows the principality of the shear map.
\end{proof}
\begin{definition}
  \label{CechNerve}
  For $Y \to X$ a morphism in $\mathrm{sSh}(C)$, write
  $$
    \check{C}(Y)
	\in [\Delta^{\mathrm{op}}, \mathrm{sSh}(C)]
  $$
  for its \emph{{\v C}ech nerve}, 
  given in degree $n$ by the $n$-fold fiber product of
  $Y$ over $X$
  $$
    \check{C}(Y)_n := Y^{\times_X^{n+1}}
	\,.
  $$
\end{definition}
\begin{observation}
  \label{CanonicalMorphismOutOfChechNerveHocolim}
  Under $\sigma_*$ the canonical morphism of simplicial objects $\check{C}(Y) \to X$, with $X$
  regarded as a constant simplicial object induces (by  
  Proposition~\ref{TotalSimpSetEquivalentToDiagonal}) 
  canonical morphism
  $$
    \sigma_* \check{C}(Y) \to X\;\; \in \mathrm{sSh}(C)
	\,.
  $$
\end{observation}
\begin{lemma}
  \label{CechNerveProjectionIsWeakEquivalence}
  For $p : Y \to X$ a local acyclic fibration, the morphism 
  $\sigma_* \check{C}(Y) \to X$ from Observation \ref{CanonicalMorphismOutOfChechNerveHocolim}
  is a local weak equivalence.
\end{lemma}
\begin{proof}
  By pullback stability of local acyclic fibrations, for each $n \in \mathbb{N}$
  the morphism $Y^{\times^n_X} \to X$ is a local weak equivalence. By
  Remark \ref{Total simplicial object is built from finite limits}
  and Proposition \ref{TotalSimpSetEquivalentToDiagonal} this degreewise local weak 
  equivalence is preserved by the functor $\sigma_*$.
\end{proof}
The main statement now is the following.
\begin{theorem}
  \label{HomotopyQuotientWPrincBundleLoAcyclicFibration}
   For $P \to X$ a weakly $G$-principal bundle in $\mathrm{sSh}(C)$, 
   the canonical morphism
  $$
    P \!/_h G  \longrightarrow X
  $$
  is a local acyclic fibration.
\end{theorem}
\begin{proof}
To see that the morphism is a local weak equivalence,
factor $P/\!/G \to X$ in $[\Delta^{\mathrm{op}}, \mathrm{sSh}(C)]$ 
via the multi-shear maps from Lemma~\ref{multishear} through the
{\v C}ech nerve, Definition~\ref{CechNerve}, as  
\[
  P/\!/G \to \check{C}(P) \to X
  \,.
\]
Applying the total simplicial object functor $\sigma_*$  
(Definition \ref{TotalSimplicialSet})
yields a factorization
$$
  P \!/_h G  \to  \sigma_* \check{C}(P)  \to X
  \,.
$$
The left morphism is a weak equivalence because, by Lemma~\ref{multishear}, 
the multi-shear maps are weak equivalences and by 
Corollary~\ref{SimplicialHomotopyColimitByCodiagonal} 
$\sigma_*$ preserves sends degreewise weak equivalences to weak equivalences. 
The right map is a weak equivalence by Lemma~\ref{CechNerveProjectionIsWeakEquivalence}.

We now prove that $P/_h G\to X$ is a local fibration.  
We need to show that for each topos point $p$ of $\mathrm{Sh}(C)$
the morphism  of stalks $p(P/_h G)\to p(X)$ is a Kan fibration of simplicial sets.
By Proposition~\ref{TotalSimplicialObjectByBorelConstruction} this means equivalently
that the morphism
$$
  p( P \times_G W G ) \to p(X) 
$$
is a Kan fibration.
By definition of \emph{topos point}, $p$ commutes with all the finite products 
and colimits involved here. Therefore equivalently we need to show that
$$
  p(P) \times_{p(G)} W p(G)  \to p(X) 
$$
is a Kan fibration for all topos points $p$.
Observe that this morphism factors the projection 
$p(P) \times W (p(G))\to p(X)$ as 
$$
  p(P)\times W(p(G))\to p(P) \times_{p(G)} W (p(G))\to p(X) 
$$
in $\mathrm{sSet}$. Here the first morphism is a Kan fibration by
Lemma~\ref{PropertiesOfSimplicialQuotientsBySimplicialGroups}, which 
in particular is also surjective on vertices. Also the total composite morphism is a Kan 
fibration, since $W (p(G))$ is Kan fibrant.  
From this the desired result follows with the next Lemma~\ref{RightCancelSurjectiveFibFromFib}.  
\end{proof}
\begin{lemma} 
  \label{RightCancelSurjectiveFibFromFib}
Suppose that $X\xrightarrow{p} Y \xrightarrow{q} Z$ is a 
diagram of simplicial sets such that $p$ is a Kan fibration surjective 
on vertices and $qp$ is a Kan 
fibration.  Then $q$ is also a Kan fibration.  
\end{lemma} 
This is Exercise V3.8 in \cite{GoerssJardine}.

\medskip

We now discuss the equivalence between weakly $G$-principal bundles and $G$-cocycles.
For $X, A \in \mathrm{sSh}(C)$, 
write $\mathrm{Cocycle}(X,A)$
for the category of cocycles from $X$ to $A$, according to \ref{GroupoidsOfMorphisms}.
\begin{definition}
 \label{FunctorsBetweenBundlesAndCocycles}
   Let $X \in \mathrm{sSh}(C)$ be locally fibrant, and let
   $G \in \mathrm{sSh}(C)$ be a group object.  
   Define a functor
   $$
     \mathrm{Extr} : \mathrm{w}G\mathrm{Bund}(X) \to \mathrm{Cocycle}(X, \Wbar G)
   $$
   (``extracting'' a cocycle) 
   on objects by sending a weakly $G$-principal bundle $P \to X$ to the cocycle  
$$
  \xymatrix{
    X \ar@{<<-}[r]^\sim & P/_h G \ar[r] & \Wbar G
  }
  \,,
$$
where the left morphism is the local acyclic fibration from 
Theorem~\ref{HomotopyQuotientWPrincBundleLoAcyclicFibration}, 
and where the right morphism
is the image under $\sigma_*$ 
(Definition~\ref{TotalSimplicialSet}) of the 
canonical morphism $P/\!/G \to */\!/G$ of simplicial objects.

Define also a functor 
\[
  \mathrm{Rec} : \mathrm{Cocycle}(X,\Wbar G) \to \mathrm{w}G\mathrm{Bund}(X) 
\]
(``reconstruction'' of the bundle)
which on objects takes a cocycle $X\xleftarrow{\pi} Y \xrightarrow{g} \Wbar G$ to the 
weakly $G$-principal bundle
$$
  g^* W G \to Y \stackrel{\pi}{\to} X
  \,,
$$
which is the pullback of the universal $G$-principal bundle 
(Definition~\ref{UniversalSimplicialPrincipalBundle})
along $g$,
and which on morphisms takes a coboundary to the morphism between pullbacks induced from
the corresponding morphism of pullback diagrams.
\end{definition}
\begin{observation}  
  \label{TheUniversalCocycle}
  The functor $\mathrm{Extr}$ sends the universal $G$-principal bundle
  $W G \to \Wbar G$ to the cocycle
  $$
    \Wbar G \simeq * \times_G W G \stackrel{\simeq}{\leftarrow}
	W G \times_G W G \stackrel{\simeq}{\to} W G \times_G * \simeq \Wbar G
	\,.
  $$
  Write 
  $$
    q : \mathrm{Cocycle}(X,  \Wbar G) \to \mathrm{Cocycle}(X,  \Wbar G)
  $$
  for the functor given by postcomposition with this universal cocycle.
  This has an evident left and right adjoint $\bar q$. Therefore 
  under the simplicial nerve these
  functors induce homotopy equivalences in $\mathrm{sSet}_{\mathrm{Quillen}}$.
\end{observation}
\begin{theorem}
 \label{Classification theorem for weakly G-principal bundles}
  The functors $\mathrm{Extr}$ and $\mathrm{Rec}$ from Definition~\ref{FunctorsBetweenBundlesAndCocycles}
  induce weak equivalences
 $$
   N \mathrm{w}G\mathrm{Bund}(X) \simeq N \mathrm{Cocycle}(X, \Wbar G)
   \;\;
   \in \mathrm{sSet}_{\mathrm{Quillen}}
 $$
 between the simplicial nerves of the category of weakly $G$-principal bundles
 and of cocycles, respectively.
\end{theorem}
\begin{proof}
We construct natural transformations
$$
  \mathrm{Extr} \circ \mathrm{Rec} \Rightarrow q
$$
and
$$
  \mathrm{Rec} \circ \mathrm{Extr} \Rightarrow \mathrm{id}
  \,,
$$
where $q$ is the homotopy equivalence from Observation~\ref{TheUniversalCocycle}.

For
$$
  X \xleftarrow{\pi} Y \xrightarrow{f} \Wbar G. 
$$
a cocycle, its image under $\mathrm{Extr} \circ \mathrm{Rec} $ is  
\[
  X \leftarrow (f^* WG) /_h G \to \Wbar G
  \,.
\] 
The morphism $(f^*WG) /_hG \to X$ factors through $Y$ by construction, 
so that the left triangle in the diagram
\[
\xymatrix@R=0.7pc {
  & & (f^*WG) /_hG \ar[dd]\ar@{->>}[lld]_\sim\ar[rrd] & & 
  \\
  X & & & & \Wbar G 
  \\
  & & Y \ar@{->>}[llu]^\sim\ar@{-->}[urr]^{q(f)} & & 
}
\]
commutes. The top right morphism is by definition the image under 
$\sigma_*$ (Definition \ref{TotalSimplicialSet}) of 
$(f^* WG) /\!/ G \to * /\!/ G$. This factors the top horizontal morphism as
$$
  \raisebox{20pt}{
  \xymatrix{
    (f^* W G) /\!/ G \ar[d] \ar[r] & (W G)/\!/G \ar[d] \ar[r] & {*} /\!/G 
	\\
	Y \ar[r]^f & \Wbar G \,.
  }
  }
$$
Applying the total simplicial object functor to this diagram gives the above commuting
triangle on the right. Clearly this construction is natural and hence provides 
a natural transformation $\mathrm{Extr} \;\mathrm{Rec} \Rightarrow q$.

For the other natural transformation, 
let now $P \to X$ be a weakly $G$-principal bundle. This induces the following
commutative diagram of simplicial objects (with $P$ and $X$ regarded as constant
simplicial objects)
$$
  \raisebox{20pt}{
  \xymatrix{
    P \ar@{<-}[r] \ar[d] & (P \times_X P) /\!/G \ar[d] &
    (P \times G) /\!/ G \ar[l]^<<<<\sim_<<<<{\phi} \ar[r]  \ar[d] &	G /\!/ G \ar[d]	  
    \\
    X \ar@{<-}[r] & P/\!/G  \ar@{=}[r] & P/\!/G \ar[r] & {*} /\!/ G
  }
  }
  \,,
$$
where the left and the right square are pullbacks, and where the 
top horizontal morphism $\phi$ is 
the degreewise local weak equivalence which is degreewise 
induced by the shear map, composed with exchange of the two factors. 


The image of the above diagram under $\sigma_*$, which 
preserves all the pullbacks and weak equivalences involved, is
$$
  \raisebox{20pt}{
  \xymatrix{
    P \ar@{<<-}[r]^<<<<<<<\sim \ar@{->>}[d] & P \times_X P /_h G \ar@{->>}[d] &
    (P \times G) /_h G \ar[l]_\sim \ar[r]  \ar@{->>}[d] &	W G \ar@{->>}[d]	  
    \\
    X \ar@{<<-}[r]^\sim & P/_hG  \ar@{=}[r] & P/_hG \ar[r] & \Wbar G
  }
  }
  \,.
$$
Here the total bottom span is the cocycle $\mathrm{Extr}(P)$, and so 
the object $(P \times G)/_h G$ over $X$ is $\mathrm{Rec}(\mathrm{Extr}(P))$. 
Therefore this exhibits a natural morphism $\mathrm{Rec} \,\mathrm{Extr} P \to P$.
\end{proof}
\begin{remark}
  \label{ClassificationTheoremRelatedToCocyclesSpaces}
  By Theorem~\ref{SimplicialLocalizationOfCatOfFibrantObjects}
  the simplicial set $N \mathrm{Cocycle}(X, \Wbar G)$ is a presentation of the 
  intrinsic cocycle $\infty$-groupoid $\mathbf{H}(X, \mathbf{B}G)$
  of the hypercomplete $\infty$-topos
 $\mathbf{H} = \mathrm{Sh}_\infty^{\mathrm{hc}}(C)$.
 Therefore the equivalence of Theorem~\ref{Classification theorem for weakly G-principal bundles}
 is a presentation of
 theorem 3.19 in \cite{NSSa},
 $$
   G \mathrm{Bund}_\infty(X) \simeq \mathbf{H}(X, \mathbf{B}G)
 $$
 between the $\infty$-groupoid of $G$-principal $\infty$-bundles in $\mathbf{H}$
  and 
 the intrinsic cocycle $\infty$-groupoid of $\mathbf{H}$.
\end{remark}
\begin{corollary}
 \label{befreiung}
For each weakly $G$-principal bundle $P \to X$ there is a weakly $G$-principal
bundle $P^{f}$ with a levelwise free $G$-action and a weak equivalence 
$P^{f} \xrightarrow\sim P$ of weakly $G$-principal bundles over $X$. 
In fact, the assignment $P \mapsto P^f$ is an homotopy inverse to the full inclusion
of weakly $G$-principal bundles with free action into all weakly $G$-principal bundles.
\end{corollary}
\begin{proof}
Note that the universal bundle $WG \to \Wbar G$ carries a free $G$-action, in the sense that the levelwise action of $G_n$ on $(W G)_n$ is free. 
This means that the functor $\mathrm{Rec}$ 
from the proof of Theorem~\ref{Classification theorem for weakly G-principal bundles} 
indeed takes values in weakly $G$-principal budles with free action.
 Hence we can set 
 $$
   P^f 
     := 
    \mathrm{Rec}(\mathrm{Extr}(P)) = (P \times G) /_h G
	\,.
 $$
 By the discussion there we have a natural morphism $P^f \to P$ and one checks
 that this exhibits the homomotopy inverse.
\end{proof}

\subsection{Associated bundles}
\label{StrucRepresentations}

In Section 4.1 of \cite{NSSa} is discussed a 
general notion of $V$-fiber bundles which are \emph{associated} to a 
$G$-principal $\infty$-bundle via an action of $G$ on some $V$.
Here we discuss presentations of these structures in terms of the
weakly principal simplicial bundles from 
Section~\ref{Principal infinity-bundles presentations}.

\medskip

Let $C$ be a site with terminal object. 
By Proposition~\ref{InftyGroupsBySimplicialGroups}
every $\infty$-group over $C$ has a presentation by a 
sheaf of simplicial groups $G \in \mathrm{Grp}(\mathrm{sSh}(C)_{\mathrm{lfib}})$. 
Moreover, by Theorem~\ref{Classification theorem for weakly G-principal bundles}
every $\infty$-action of $G$ on an object $V$
according to Definition 3.1 of \cite{NSSa},
is exhibited by a weakly principal simplicial bundle $V \to V/_h G$
which is classified by a morphism $\mathbf{c} : V/_h \to \Wbar G$.
The resulting fiber sequence of simplicial presheaves
$$
  \raisebox{20pt}{
  \xymatrix{
    V \ar[r] & V /_h G
	  \ar[d]^{\mathbf{c}}
	\\
	& \Wbar G
  }
  }
$$
is therefore a presentation for the
\emph{universal $\rho$-associated $V$-bundle} from Section 4.1 of \cite{NSSa}.

In terms of this presentation, Proposition 4.6 in \cite{NSSa}
has the following ``strictification''.
\begin{proposition}
  Let $P \to X$ in $\mathrm{sSh}(C)_{\mathrm{lfib}}$ be a weakly $G$-principal bundle
  with classifying cocycle 
  $\xymatrix{
     X \ar@{<-}[r]^{\simeq} & Y \ar[r] & \Wbar G
  }$
  according to Theorem \ref{Classification theorem for weakly G-principal bundles}.
  Then the $\rho$-associated simplicial $V$-bundle $P \times_G V$ is 
  locally weakly equivalent to the pullback of $\mathbf{c}$ along $g$.
\end{proposition}
\proof
  By the same argument as in
  the proof of Theorem~\ref{HomotopyQuotientWPrincBundleLoAcyclicFibration},
  the morphism $\mathbf{c} : V /_h G \to \Wbar G$
  is a local fibration.
   By Proposition~\ref{TotalSimplicialObjectByBorelConstruction}
  this in turn is isomorphic to the pullback of
  $V \times_G W G \to \Wbar G$. Since $\mathrm{sSh}(C)$ is a 1-topos, 
  pullbacks preserve quotients, and so this pullback finally is
  $$
    g^* (W G \times_G V) \simeq (g^* W G) \times_G V \simeq P \times_G V
	\,.
  $$  
\endofproof
\begin{remark}
According to Theorem 4.11 in \cite{NSSa}, every $V$-fiber bundle in an $\infty$-topos
is associated to an $\mathbf{Aut}(V)$-principal $\infty$-bundle. We observe that 
the main result of \cite{Wendt} is a presentation of this general theorem 
for 1-localic $\infty$-toposes (with a 1-site of definition) in terms of simplicial
presheaves. 
\end{remark}

\section{Models}
\label{Models}

So far we have discussed presentations of the theory of principal $\infty$-bundles
over arbitrary sites. Here we consider certain examples of sites and discuss
aspects of the resulting presentations.

\begin{itemize}
  \item 
    The trivial site models higher \emph{discrete geometry}. We show how in this case
    the general theory reduces to the classical theory of ordinary simplicial 
	principal bundles in Section \ref{DiscreteGeometry}.
 \item 
   The site of smooth manifolds models higher
   \emph{smooth geometry}/\emph{differential geometry}. 
   Since this site does not have all pullbacks, 
   item 2 of Proposition \ref{DegreewiseRepresentability} does not
   apply, and so it is of interest to identify conditions under which 
   a given principal $\infty$-bundle is presentable not just by 
   a simplicial smooth manifold, but by a \emph{locally Kan}
   simplicial smooth manifolds. This we discuss in  Section \ref{SmoothInfgrpds}.
\end{itemize}

\subsection{Discrete geometry}
\label{DiscreteGeometry}
\label{PrincipalBundlesDiscreteGeometry}

The terminal $\infty$-topos is the $\infty$-category
$\mathrm{Grpd}_{\infty}$ of $\infty$-groupoids, the one presented by the 
standard model category structures on simplicial sets and on topological spaces.
Regarded as a \emph{gros} $\infty$-topos akin to that of 
smooth $\infty$-groupoids discussed below in \ref{SmoothInfgrpds},
we are to think of $\mathrm{Grpd}_{\infty}$ as describing \emph{discrete} geometry: 
an object in $\mathrm{Grpd}_{\infty}$ is an $\infty$-groupoid without extra 
geometric structure. In order to amplify this geometric perspective, 
we will sometimes speak of \emph{discrete $\infty$-groupoids}. 

\medskip

We have $\mathrm{Grpd}_{\infty} \simeq \mathrm{Sh}_\infty(*)$, for $*$ the trivial
site. For this site the category of locally fibrant simplicial sheaves from Observation
\ref{SimplicialSheavesWithStalkwiseWeakEquivalencesModel1LocalicHypercompleteInfinityTopos}
is equivalent simply to the category of Kan complexes
$$
  \xymatrix{
    \mathrm{sSh}(*)_{\mathrm{lfib}}\  \ar@{^{(}->}[r] \ar[d]^\simeq & \ \mathrm{sSh}(*) \ar[d]^\simeq
	\\
	\mathrm{KanCpx}\  \ar@{^{(}->}[r] &\  \mathrm{sSet}
  }
$$
and local fibrations/equivalences are simply Kan fibrations and weak homotopy equivalences
of simplicial sets respectively. A group object $G$ in $\mathrm{sSh}(C)_{\mathrm{lfib}}$ is 
a simplicial group; therefore over the trivial site the presentation of
principal $\infty$-bundles from \ref{Principal infinity-bundles presentations}
is by weakly principal Kan simplicial bundles.

There is a traditional theory of \emph{strictly} principal Kan simplicial bundles,
i.e.\  simplicial bundles with $G$ action for which the shear map is an \emph{isomorphism}
instead of, more generally, a weak equivalence, see also 
Remark \ref{PrincipalityOfUniversalSimplicialBundle}.
A classical reference for this is \cite{May}. A standard modern reference is
Chapter V of \cite{GoerssJardine}. We now compare this classical theory 
of strictly principal simplicial bundles to the theory of weakly principal
simplicial bundles according to Section \ref{Principal infinity-bundles presentations}.

\medskip

\begin{definition}
\label{strictly principal bundle}
  Let $G$ be a simplicial group and $X$ a Kan simplicial set. A
  \emph{strictly $G$-principal bundle} over $X$ is a morphism of simplicial 
  sets $P \to X$ equipped with a $G$-action on $P$ over $X$ such that
  \begin{enumerate}
    \item the $G$ action is degreewise free;
	\item the canonical morphism $P/G \to X$ out of the ordinary (1-categorical)
	  quotient is an isomorphism of simplicial sets.
  \end{enumerate}
  A morphism of strictly $G$-principal bundles over $X$ is a map $P \to P'$
  respecting both the $G$-action as well as the projection to $X$.  
  Write $\mathrm{s}G\mathrm{Bund}(X)$ for the category of strictly $G$-principal bundles.
\end{definition}
In \cite{GoerssJardine} this is Definitions 3.1 and 3.2 of Chapter V.
\begin{lemma}
  \label{MorphismsOfStrictlyPrincipalSimplicialBundlesAreIsos}
  Every morphism in $\mathrm{s}G\mathrm{Bund}(X)$ is an isomorphism.
\end{lemma}
In \cite{GoerssJardine} this is Remark 3.3 of Chapter V.
\begin{observation}
  \label{InclusionOfStrictlyPrincipalSimplicialBundles}
  Evidently every strictly $G$-principal bundle is also a weakly $G$-principal
  bundle, Definition~\ref{WeaklyGPrincipalBundle}; in fact the strictly principal $G$-bundles
  are precisely those weakly $G$-principal bundles for which the shear map is an isomorphism.
  This identification induces a full inclusion of categories
  $$
    \mathrm{s}G\mathrm{Bund}(X) \hookrightarrow \mathrm{w}G\mathrm{Bund}(X)
	\,.
  $$
\end{observation}
\begin{lemma}
  \label{MorphismsOfWeaklyPrincipalSimplicialBundlesAreWeakEquivalences}
  Every morphism of weakly principal simplicial bundles in $\mathrm{KanCpx}$ 
  is a weak homotopy 
  equivalence on the underlying Kan complexes.
\end{lemma}
\begin{proposition}
  \label{ModelStructureOnGSimplicialSets}
  For $G$ a simplicial group, the category $\mathrm{sSet}_G$ of $G$-actions on
  simplicial sets and $G$-equivariant morphisms carries the structure of a 
  simplicial model category where the fibrations and weak equivalences are
  those of the underlying simplicial sets.
\end{proposition}
This is Theorem 2.3 of Chapter V in \cite{GoerssJardine}.
\begin{corollary}
  \label{SliceModelStructureOnGSimplicialSetsOverX}
  For $G$ a simplicial group and $X$ a Kan complex, the slice category
  $\mathrm{sSet}_G/X$ carries a simplicial model structure where the fibrations
  and weak equivalences are those of the underlying simplicial sets, after
  forgetting the map to $X$.
\end{corollary}
\begin{lemma}
  \label{WeaklyPrincipalSimplicialBundlesSitHomotopyFFInModelStructureOnGSimplicialSets}
  Let $G$ be a simplicial group and $P \to X$ a weakly $G$-principal simplicial bundle.
  Then the loop space $\Omega_{(P \to X)} \mathrm{Ex}^\infty N (\mathrm{w}G\mathrm{Bund}(X))$   
  has the same homotopy type as the derived hom space
  $\mathbb{R}\mathrm{Hom}_{\mathrm{sSet}_G/X}(P,P)$.
\end{lemma}
\proof
  By Theorem 2.3, Chapter V of \cite{GoerssJardine} and 
  Lemma~\ref{MorphismsOfWeaklyPrincipalSimplicialBundlesAreWeakEquivalences} 
  the free resolution $P^f$ of $P$
  from Corollary~\ref{befreiung} is a cofibrant-fibrant resolution of $P$ in 
  the slice model structure of Corollary~\ref{SliceModelStructureOnGSimplicialSetsOverX}.
  Therefore the derived hom space is presented by the simplicial set
  of morphisms $\mathrm{Hom}_{\mathrm{sSet}_G/X}(P^f \cdot \Delta^\bullet, P^f)$
  and all these morphisms are equivalences. Therefore by Proposition 2.3 in 
  \cite{DwyerKanClassification} this simplicial set is equivalent to the 
  loop space of the nerve of the subcategory of $\mathrm{sSet}_G/X$ on the 
  weak equivalences connected to $P^f$. 
  By Lemma~\ref{MorphismsOfWeaklyPrincipalSimplicialBundlesAreWeakEquivalences}
  this subcategory is 
  equivalent (isomorphic even) to the connected component of $\mathrm{w}G\mathrm{Bund}(X)$
  on $P$.
\endofproof
\begin{proposition}
  \label{StrictlyPrincipalBundlesSurjectOnPi0}
  Under the nerve,  the inclusion of 
  observation \ref{InclusionOfStrictlyPrincipalSimplicialBundles} yields a morphism
  $$
    N \mathrm{s}G\mathrm{Bund}(X) \to N \mathrm{w}G\mathrm{Bund}(X)
  $$
  which is 
  \begin{itemize}
    \item for all $G$ and $X$ an isomorphism on connected components;
	\item not in general a weak equivalence in $\mathrm{sSet}_{\mathrm{Quillen}}$.
  \end{itemize}
\end{proposition}
\proof
  Let $P \to X$ be a weakly $G$-principal bundle. 
  To see that it is connected in $\mathrm{w}G\mathrm{Bund}(X)$ to some
  strictly $G$-principal bundle, first observe that by Corollary~\ref{befreiung} it is 
  connected via a morphism $P^f \to P$ to the bundle 
  $$
    P^f := \mathrm{Rec}(X \leftarrow P/_h G \xrightarrow{f} \Wbar G)
	\,,
  $$
  which has free $G$-action, but does not necessarily satisfy strict principality.
  Since, by Theorem~\ref{HomotopyQuotientWPrincBundleLoAcyclicFibration}, 
  the morphism $P/_h G \to X$ is an acyclic fibration of simplicial sets
  it has a section $\sigma : X \to P /_h G$ (every simplicial set
  is cofibrant in $\mathrm{sSet}_{\mathrm{Quillen}}$). The bundle 
  $$
    P^s := 
	  \mathrm{Rec}(X \xleftarrow{\mathrm{id}} X \xrightarrow{f\circ \sigma} \Wbar G)
  $$
  is strictly $G$-principal, and with the morphism
  $$
    (P^s \to P^f)
	:=
    \mathrm{Rec}
	\left(
	   \raisebox{44pt}{
	   \xymatrix{
	     & P /_h G
		 \ar@{->>}[dl]_\sim \ar[dr]^f
	     \\ 
	     X && \Wbar G
		 \\
		 & X \ar@{->>}[ul]^{\mathrm{id}} \ar[uu]_\sigma \ar[ur]_{f \circ \sigma}
	   }
	   }
	\right)
  $$
  we obtain (non-naturally, due to the choice of section) in total a morphism
  $P^s \to P^f \to P$
  of weakly $G$-principal bundles from a strictly $G$-principal replacement $P^s$ to $P$.
  
  To see that the full embedding of strictly $G$-principal bundles is also injective on connected
  components, notice that by Lemma~\ref{WeaklyPrincipalSimplicialBundlesSitHomotopyFFInModelStructureOnGSimplicialSets}
  if a weakly $G$-principal bundle $P$ with degreewise free $G$-action is connected
  by a zig-zag of morphisms to some other weakly $G$-principal bundle $P$, then there is 
  already a direct morphism $P \to P'$. Since all strictly $G$-principal bundles
  have free actions by definition, this shows that two of them that are
  connected in $\mathrm{w}G\mathrm{Bund}(X)$ are already connected in 
  $\mathrm{s}G\mathrm{Bund}(X)$.
  
  To see that in general $N\mathrm{s}G\mathrm{Bund}(X)$ 
  nevertheless does not have the 
  correct homotopy type,  it is sufficient to notice that 
  the category $\mathrm{s}G\mathrm{Bund}(X)$ is always a groupoid,
  by Lemma~\ref{MorphismsOfStrictlyPrincipalSimplicialBundlesAreIsos}. Therefore 
  $N\mathrm{s}G\mathrm{Bund}(X)$ is always a homotopy 1-type. 
  But by Theorem~\ref{Classification theorem for weakly G-principal bundles}
  the object $N\mathrm{w}G\mathrm{Bund}(X)$ is not an $n$-type if $G$ is not an
  $(n-1)$-type. 
\endofproof
\begin{corollary}
 \label{StrictlyPrincipalSimplicialBundlesModelCohomologyClasses}
 For all Kan complexes $X$ and simplicial groups $G$ there is an isomorphism
 $$
	\pi_0 N \mathrm{s}G\mathrm{Bund} \simeq  H^1(X,G) := \pi_0 \mathrm{Grpd}_{\infty}(X, \mathbf{B}G)
 $$
 between the isomorphism classes of strictly $G$-principal bundles over $X$ and
 the first nonabelian cohomology of $X$ with coefficients in $G$ (but this 
 isomorphism on cohomology does not in general lift to an equivalence
  on cocycle spaces).
\end{corollary}
\proof
  By Proposition~\ref{StrictlyPrincipalBundlesSurjectOnPi0} and Remark~\ref{ClassificationTheoremRelatedToCocyclesSpaces}.
\endofproof
\begin{remark}
The first statement of corollary \ref{StrictlyPrincipalSimplicialBundlesModelCohomologyClasses} 
is the classical classification result for strictly principal simplicial bundles,
for instance Theorem V3.9 in \cite{GoerssJardine}. 
\end{remark}

\subsection{Smooth geometry}
\label{SmoothInfgrpds}

We discuss the canonical homotopy theoretic context for higher
differential geometry.

\begin{definition}
Let $\mathrm{SmthMfd}$ be the category of finite dimensional 
smooth manifolds. We regard this as a site with the covers being the
standard open covers. 
Write 
$$
  \mathrm{CartSp} \hookrightarrow \mathrm{SmthMfd}
$$ 
for the full subcategory on the Cartesian spaces $\mathbb{R}^n$, for all $n \in \mathbb{N}$,
equipped with their canonical structure of smooth manifolds.
\end{definition}
\begin{proposition}
  The inclusion $\mathrm{CartSp} \hookrightarrow \mathrm{SmthMfd}$
  exhibits $\mathrm{CartSp}$ as a \emph{dense subsite} of $\mathrm{SmthMfd}$.
  Accordingly, there is an equivalence of categories between the sheaf toposes
  over both sites,
  $\mathrm{Sh}(\mathrm{CartSp}) \simeq \mathrm{Sh}(\mathrm{SmthMfd})$.
\end{proposition}
\begin{lemma}
  \label{ToposPointsOfSheavesOverSmoothManifolds}
  The sheaf topos $\mathrm{Sh}(\mathrm{CartSp})\simeq \mathrm{Sh}(\mathrm{SmthMfd})$ 
  has enough points. 
  A complete set of points
  $$
    \left\{
    \xymatrix{
	   \mathrm{Set}
	     \ar@{<-}@<+3pt>[r]^<<<<<{p^n}
		 \ar@<-3pt>[r]
		 &
	   \mathrm{Sh}(\mathrm{CartSp})
	}
	\;\Big|\; n \in \mathbb{N}
	\right\}
  $$
  is given by the stalks at the origin of the
  open $n$-disk, for all $n \in \mathbb{N}$.
\end{lemma} 
This  statement was first highlighted in \cite{DuggerSheavesAndHomotopy}.
  In more detail, $p_n$ is given as follows. 
  Let $D^n_k\subset \RR^n$ denote the smooth manifold given by the standard 
open $n$-disk of radius $1/k$ centered at the origin.    
  For $X \in \mathrm{Sh}(\mathrm{CartSp})$ and $n \in \mathbb{N}$ the 
  \emph{$n$-stalk} of $X$ is the colimit 
  $$
     p_n(X) =   \varinjlim_{k\to \infty} \mathrm{Hom}(D^n_k,X). 
  $$
  of the values on $X$ on these disks.
  In particular the set $p_0(X)$ is the set of global sections of $X$.

\begin{definition}
  The $\infty$-topos of \emph{smooth $\infty$-groupoids} is
  $$
    \mathrm{Smooth}\mathrm{Grpd}_{\infty} := \mathrm{Sh}_{\infty}(\mathrm{CartSp})
	\,.
  $$
\end{definition}
\begin{proposition}
The $\infty$-topos $\mathrm{Smooth}\mathrm{Grpd}_{\infty}$ has the following properties:
\begin{enumerate}
 \item It is hypercomplete.
 \item It is equivalent to $\mathrm{Sh}_\infty(\mathrm{SmthMfd})$.
 \item The site $C$ is a \emph{$\infty$-cohesive site} (Definition \ref{CohesiveSite}).
\end{enumerate}
\end{proposition}
These and the following statements are discussed in detail in Section 4.4. of \cite{Schreiber}.
In particular we have
\begin{observation}
  The $\infty$-topos of smooth $\infty$-groupoids is presented by the local projective
  model structure on simplicial presheaves over $\mathrm{CartSp}$
  $$
    \mathrm{Smooth}\mathrm{Grpd}_\infty \simeq ([\mathrm{CartSp}^{\mathrm{op}}, \mathrm{sSet}]_{\mathrm{proj}, \mathrm{loc}})
	\,.
  $$
  For 
  $$
    X \in \mathrm{SmthMfd} \hookrightarrow [\mathrm{CartSp}^{\mathrm{op}}, \mathrm{sSet}]
  $$
  a smooth manifold and $\{U_i \to U\}$ a \emph{good open cover} in the sense that every
  non-empty finite intersection of the $U_i$ is diffeomorphic to an open ball, the
  {\v C}ech nerve $\check{C}(\{U_i \}) \to X$ is a split hypercover. Hence every
  morphism out of 
  $X \in \mathrm{Smooth}\mathrm{Grpd}_\infty$ is presented by a hyper-{\v C}ech cocycle
  with respect to this cover.
  \label{ResolutionOfManifoldsByGoodCovers}
\end{observation}
\begin{definition}
  \label{LocallyFibrantSimplicialSheavesOnSmoothManifolds}
  Write
  $$
    \mathrm{sSh}(\mathrm{CartSp}) := [\Delta^{\mathrm{op}}, \mathrm{Sh}(\mathrm{CartSp})]
  $$
  for the category of simplicial objects in the sheaf topos over $\mathrm{CartSp}$.
  As in Section~\ref{InfinityToposPresentation}, 
  we say that a morphism in $\mathrm{sSh}(\mathrm{CartSp})$ is 
  \begin{itemize} 
\item a {\em local weak equivalence} if it is stalkwise a weak equivalence of simplicial sets;

\item  a {\em local fibration} if it is stalkwise a Kan fibration of simplicial sets,
  \end{itemize}
   where the stalks $\{p_n\}_{n \in \mathbb{N}}$ are those of 
   Lemma~\ref{ToposPointsOfSheavesOverSmoothManifolds}. Write
   $$
     \mathrm{sSh}(\mathrm{CartSp})_{\mathrm{lfib}}
	 \hookrightarrow
	 \mathrm{sSh}(\mathrm{CartSp})
   $$
   for the full subcategory on the locally fibrant objects.
\end{definition}
\begin{proposition}
  \label{Presentations of Smooth infinity-groupoids by locally fibrant simplicial sheaves}
  The $\infty$-topos $\mathrm{Smooth}\mathrm{Grpd}_{\infty}$ is presented 
  by the
  category $\mathrm{sSh}(\mathrm{CartSp})_{\mathrm{lfib}}$ from 
  Definition~\ref{LocallyFibrantSimplicialSheavesOnSmoothManifolds} with weak equivalences
  the local weak equivalences
  $$
    \mathrm{Smooth}\mathrm{Grpd}_{\infty} \simeq L_W \mathrm{sSh}(\mathrm{CartSp})_{\mathrm{lfib}}
	\,.
  $$
  Together with the local fibrations this is 
  a category of fibrant objects, Definition~\ref{CategoryOfFibrantObjects}.
\end{proposition}
  Therefore the hom-$\infty$-groupoids are equivalently given by the 
  cocycle categories of Proposition~\ref{SimplicialLocalizationOfCatOfFibrantObjects}.

\subsubsection{Locally fibrant simplicial manifolds}
\label{Locally fibrant simplicial smooth manifolds}

By Proposition~\ref{Presentations of Smooth infinity-groupoids by locally fibrant simplicial sheaves} 
smooth $\infty$-groupoids
are presented by locally fibrant simplicial sheaves on $\mathrm{CartSp}$. Every
simplicial manifold represents a simplicial sheaf over this site.
We discuss now the full sub-$\infty$-category of $\mathrm{Smooth}\mathrm{Grpd}_{\infty}$
on those objects that are presented by locally Kan fibrant simplicial smooth manifolds.

\medskip

\begin{definition}
  \label{LocallyFibrantSimplicialSmoothManifolds}
  Let the category of \emph{locally fibrant simplicial smooth manifolds}
  be the full subcategory 
  $$
    \mathrm{sSmthMfd}_{\mathrm{lfib}}
	\hookrightarrow
	\mathrm{sSh}(\mathrm{CartSp})_{\mathrm{lfib}} 
  $$
  of the category of locally fibrant simplicial sheaves over smooth manifolds,
  Definition~\ref{LocallyFibrantSimplicialSheavesOnSmoothManifolds}, 
  on those simplicial sheaves that are represented by simplicial smooth manifolds.
\end{definition}
The structure of a category of fibrant objects on $\mathrm{sSh}(\mathrm{CartSp})_{\mathrm{lfib}}$,
Proposition~\ref{Presentations of Smooth infinity-groupoids by locally fibrant simplicial sheaves},
does not quite transfer along this inclusion, because pullbacks in $\mathrm{SmthMfd}$
do not generally exist. Pullbacks in $\mathrm{SmthMfd}$ do however exist, notably, along surjective submersions.  

Following \cite{Henriques} we will take advantage of this last fact and give the 
following enhanced definition of the 
notion of local fibration between smooth simplicial manifolds.  Before we 
do this however we briefly review the notion of matching object 
for simplicial objects in $\mathrm{SmthMfd}$.  Recall (see for example 
\cite{GoerssJardine} Section VIII) that if $X$ is a simplicial object in $\mathrm{SmthMfd}$ 
and $K$ is a simplicial set, then the limit 
\[
\varprojlim_{\Delta^n\to K} X_n 
\]
in $\mathrm{SmthMfd}$ (if it exists) is denoted $M_KX$ and is called the {\em (generalized) matching 
object} of $X$ at $K$.  Here the limit is taken over the simplex category $\Delta/K$ of 
$K$.  This notion has a straightforward generalization to simplicial objects in an 
arbitrary category $\mathcal{C}$.  To talk about matching objects we need to confront the afore-mentioned problem 
that $\mathrm{SmthMfd}$ does not have all of the limits that one would like --- very often 
we would like to talk about the limit $M_KX$ without knowing that it actually exists.  
In this situation, we will (as in \cite{Henriques}) interpret $M_KX$ as the matching 
object of the simplicial sheaf on $\mathrm{SmthMfd}$ represented by $X$.  If the 
sheaf $M_KX$ is representable then the matching objects exists in $\mathrm{SmthMfd}$.  
With these remarks out of the way we can state the following definition.

\begin{definition}  
 \label{SubmersiveLocalFibration}
 A morphism $f : X \to Y$ in $\mathrm{sSmthMfd}$ is 
 \begin{itemize}
   \item a  {\em submersive local fibration} if $f_0\colon X_0\to Y_0$ is a surjective 
submersion and for all $0\leq k\leq n$, $n\geq 1$ the canonical morphism 
\[
X_n\to Y_n\times_{M_{\Lambda^k[n]}Y}M_{\Lambda^k[n]}X 
\]
is a surjective submersion; 
 \item 
   a \emph{submersion} if $f_n : X_n \to Y_n$ is a submersion for each $n \in \mathbb{N}$;
  \item
     A simplicial smooth manifold $X$ is said to be a \emph{Lie $\infty$-groupoid} 
     if $X \to *$ is a submersive local
	 fibration and all of the face maps of $X$ are submersions.  
 \end{itemize}
\end{definition}     
\begin{example}
  \label{CechNerveProjectionOfOpenCoverIsSubmersiveLocalFibration}
  Let $X$ be a smooth manifold and $\{U_i \to X\}$ an open cover. Then the
  {\v C}ech nerve projection
  $\check{C}(\{U_i\}) \to X$ is a submersive local acyclic fibration between locally
  fibrant simplicial smooth manifolds.
\end{example}

\begin{lemma}[\cite{Henriques}] 
\label{henriques lemma}
Let $A\hookrightarrow B$ be an acyclic cofibration between finite simplicial sets.  Suppose that 
$f\colon X\to Y$ is a submersive local fibration and that 
\[
M_AX\times_{M_AY} M_BY 
\]
is a manifold.  Then $M_BX$ is a manifold and 
\[
M_B X\to M_AX\times_{M_AY} M_BY 
\]
is a surjective submersion.  
\end{lemma} 

As a corollary we have the following statement.  
\begin{corollary} 
  \label{SubmersiveLocalFibrationsAreSurjectiveSubmersions}
If $f\colon X\to Y$ is a submersive local fibration, then $f$ is a surjective submersion.  
\end{corollary} 
\begin{proof} 
Consider the acyclic cofibration $\Delta[0]\subset \Delta[n]$ corresponding to the vertex 
$0$ of $\Delta[n]$.  Then the diagram 
\[
\xymatrix{ 
X_0\times_{Y_0} Y_n \ar[d] \ar[r] & X_0 \ar[d] \\ 
Y_n \ar[r] & Y_0 } 
\]
is a pullback in $\mathrm{SmthMfd}$ and $X_0\times_{Y_0}Y_n\to Y_n$ is a surjective submersion.  
From Lemma~\ref{henriques lemma} we see that $X_n\to  X_0\times_{Y_0}Y_n$ 
is a surjective submersion.  Since the map $f_n\colon X_n\to Y_n$ factors through 
$X_0\times_{Y_0}Y_n$ we see that $f_n$ is a surjective submersion.    
\end{proof} 
%
\begin{proposition}
  \label{SubmersiveLocalFibrationsAreStableUnderPulback}
The pullback of a (locally acyclic) submersive local fibration in $\mathrm{sSmthMfd}$
exists and is again a (locally acyclic) submersive local fibration.
\end{proposition}
\begin{proof}
Suppose that $p\colon X\to Y$ is a submersive local fibration.  
Then by Lemma~\ref{SubmersiveLocalFibrationsAreSurjectiveSubmersions} 
$p_n\colon X_n\to Y_n$ is 
a surjective submersion for all $n$ and hence the pullback $X\times_Y Z$ exists in 
$\mathrm{sSmthMfd}$.  
We need to show that the projection $X\times_Y Z\to Z$ is a submersive local fibration.  
Clearly $X_0\times_{Y_0}Z_0\to Z_0$ is a submersion.  
Next, observe that we have isomorphisms of topological spaces 
\begin{align*}
M_{\Lambda^k[n]}(X\times_Y Z)\times_{M_{\Lambda^k[n]}Z} Z_n & = 
M_{\Lambda^k[n]}X\times_{M_{\Lambda^k[n]}Y}Z_n \\ 
& = (M_{\Lambda^k[n]}X\times_{M_{\Lambda^k[n]}Y}Y_n)\times_{Y_n} Z_n.  
\end{align*} 
Since the surjective submersion $X_n\to Y_n$ factors as 
\[
X_n\to M_{\Lambda^k[n]}X\times_{M_{\Lambda^k[n]}Y}Y_n \to Y_n 
\]
and $X\to Y$ is a submersive local fibration, it follows that 
\[
M_{\Lambda^k[n]}X\times_{M_{\Lambda^k[n]}Y}Y_n \to Y_n
\]
is also a surjective submersion.  Hence 
\[
M_{\Lambda^k[n]}(X\times_Y Z)\times_{M_{\Lambda^k[n]}Z} Z_n =  
(M_{\Lambda^k[n]}X\times_{M_{\Lambda^k[n]}Y}Y_n)\times_{Y_n} Z_n
\]
is a manifold and 
\[
X_n\times_{Y_n} Z_n\to M_{\Lambda^k[n]}(X\times_Y Z)\times_{M_{\Lambda^k[n]}Z} Z_n
\]
is a surjective submersion.  

To check the statement about local weak equivalences, 
use the facts that stalks commute with pullbacks and that acyclic fibrations in 
$\sSet$ are stable under pullback.  
\end{proof}

\subsubsection{Groups}

By Theorem~\ref{SimplicialLoopingQuillenEquivalence} every $\infty$-group
in $\mathrm{Smooth}\mathrm{Grpd}_{\infty}$ is presented by 
some group object in $\mathrm{sSh}(\mathrm{CartSp})$. In view 
of the discussion in Section~\ref{Locally fibrant simplicial smooth manifolds} it is of interest to 
determine those which are in the inclusion 
$\mathrm{sSmthMfd}_{\mathrm{lfib}} \hookrightarrow \mathrm{sSh}(\mathrm{CartSp})$
from Definition~\ref{LocallyFibrantSimplicialSmoothManifolds}.

\medskip

\begin{proposition} 
\label{Lie oo-grps fib}
Let $G$ be a simplicial Lie group.  Then $G$ is a Lie $\infty$-groupoid, 
and so in particular is a locally fibrant simplicial smooth manifold,
Definition~\ref{LocallyFibrantSimplicialSmoothManifolds}.  
\end{proposition}

\begin{proof} 
Clearly all of the face maps of $G$ are surjective submersions.  Therefore we need 
to prove that $G$ is locally fibrant.  Our proof is based on the observation in 
Lemma 3.3 of \cite{Stevenson2} that for any smooth manifold $Y$, 
the simplicial set $\mathrm{SmthMfd}(Y,G)$ whose set of $n$-simplices 
is the set $\mathrm{SmthMfd}(Y,G_n)$ has the structure of a simplicial 
group and hence the various maps 
\[
\mathrm{SmthMfd}(Y,G_n)\to M_{\Lambda^n_k}\mathrm{SmthMfd}(Y,G) 
\]
are all surjective.  Therefore, if the limit $M_{\Lambda^n_k}G$ 
exists in $\mathrm{SmthMfd}$ then we can identify 
\[
M_{\Lambda^n_k}\mathrm{SmthMfd}(Y,G) = \mathrm{SmthMfd}(Y,M_{\Lambda^n_k}G)
\] 
and hence conclude that 
\[
G_n\to M_{\Lambda^n_k}G
\]
admits a global section and hence is a surjective submersion.  The details are 
more delicate here than in \cite{Stevenson2} since we need to show that all of the 
requisite limits exist in $\mathrm{SmthMfd}$.  

When $n=1$ we need to show that the two face maps $d_0,d_1\colon G_1\to G_0$ 
are surjective submersions which is again clear since $s_0\colon G_0\to G_1$ is a 
global section of both of these maps.  When $n=2$ the matching objects 
$M_{\Lambda^2_k}G$ for $0\leq k\leq 2$ can be identified with pullbacks $G_1\times_{G_0}G_1$ 
which exist in $\mathrm{SmthMfd}$ since $d_0,d_1\colon G_1\to G_0$ are submersions.  
The Yoneda argument above then shows that $G_2\to M_{\Lambda^2_k}G$ is a 
surjective submersion in these cases.  

The case $n=3$ makes the general pattern clear: in this case any of the matching objects 
$M_{\Lambda^3_k}G$ for $0\leq k\leq 3$ can be identified with pullbacks of the form 
\[
G_2\times_{G_1}G_2\times_{G_1\times_{G_0}G_1}G_2
\]
in which the map $G_2\to G_1\times_{G_0}G_1$ is the canonical map $G_2\to M_{\Lambda^2_{k}}G$.  
Likewise the pullback $G_2\times_{G_1}G_2$ is the matching object $M_{\Lambda^2_k}\Dec_0G$ 
where $\Dec_0G$ is the simplicial Lie group which is the d\'{e}calage of $G$ 
(Definition~\ref{Decalage}).    

This observation forms the basis for a proof by induction on $n\geq 1$ that for 
any simplicial Lie group $G$, and any integer $0\leq k\leq n$, the limit $M_{\Lambda^n_k}G$ 
exists in $\mathrm{SmthMfd}$ (the 
Yoneda argument above then shows that $G$ is locally fibrant).  The case $n=1$ is 
clear.  

For the inductive step, first suppose that $k<n$.  
Observe that we have an identification 
\[
\Lambda^n_k = C\Lambda^{n-1}_{k}\cup_{\Lambda^{n-1}_{k}}\Delta^{n-1} 
\]
where $C\Lambda^{n-1}_{k}$ denotes the usual cone construction on $\Lambda^{n-1}_{k}$ (see 
\cite{GoerssJardine} Chapter III).  
It follows, using Corollary 2.2 of \cite{Stevenson2} and the fact that the matching objects functor 
$M_{(-)}G\colon s\Set^\op\to \mathrm{Sh}(\mathrm{CartSp})$ 
on the representable simplicial sheaf $G$ preserves limits, that the diagram 
\[
\xymatrix{ 
M_{\Lambda^n_k}G \ar[d] \ar[r] & G_{n-1} \ar[d] \\ 
M_{\Lambda^{n-1}_{k}}\Dec_0 G \ar[r] & M_{\Lambda^{n-1}_{k}}G } 
\] 
in $\mathrm{Sh}(\mathrm{CartSp})$ is a pullback.  
Hence $M_{\Lambda^n_k}G$ acquires the unique structure of a smooth manifold 
so that this diagram is a pullback in the category $\mathrm{SmthMfd}$.  It follows that 
with this unique smooth structure $M_{\Lambda^n_k}G$ is a model for the corresponding 
limit in $\mathrm{SmthMfd}$.  

For the case $k=n$ we apply the statement just 
proven with $G$ replaced by its {\em opposite} simplicial Lie group $G^o$; 
this has the property that $M_{\Lambda^n_0}G^o = M_{\Lambda^n_{n}}G$, 
which shows that the limit $M_{\Lambda^n_n}G$ exists, completing the inductive step.            
\end{proof}

\subsubsection{Principal bundles}

By the discussion in \ref{Principal infinity-bundles presentations} 
and using 
Proposition~\ref{Presentations of Smooth infinity-groupoids by locally fibrant simplicial sheaves} 
we have a presentation of principal $\infty$-bundles
in the $\infty$-topos $\mathrm{Smooth}\mathrm{Grpd}_{\infty}$ by weakly principal bundles in 
the category $\mathrm{sSh}(\mathrm{CartSp})_{\mathrm{lfib}}$ of locally fibrant
simplicial sheaves. Here we discuss how parts of this construction may be restricted
further along the inclusion 
$\mathrm{sSmthMfd}_{\mathrm{lfib}} \hookrightarrow \mathrm{sSh}(\mathrm{CartSp})_{\mathrm{locfib}}$
of locally fibrant simplicial smooth manifolds,  
Definition~\ref{Locally fibrant simplicial smooth manifolds}.

\medskip

\begin{proposition}
\label{Wbar G submersively locally fibrant}
  Let $G$ be a simplicial lie group.  Then the following statements 
  are true: 
  \begin{enumerate} 
  \item  the object $\Wbar G \in \mathrm{sSh}(\mathrm{CartSp})$,
  Definition~\ref{BarWAsCompositeWithTotal}, is presented by a submersively locally fibrant
  simplicial smooth manifold.
\item   
   the universal $G$-principal bundle
  $W G \to \Wbar G$, Definition~\ref{WGToWbarG}, formed
  in $\mathrm{sSh}(\mathrm{CartSp})$  is presented by a 
  \emph{submersive} local fibration of simplicial smooth manifolds.
  \end{enumerate}
\end{proposition}

 \begin{proof}
We first prove 1.  Our proof of this essentially follow the proof of the corresponding result (Lemma 4.3) in \cite{Stevenson2}, 
some extra care is needed however since it is not immediately clear 
that all of the requisite limits exist in $\mathrm{SmthMfd}$.  
Therefore we will prove by induction on $n\geq 1$ that for any simplicial 
Lie group $G$ and any integer $0\leq k\leq n$,  the limit 
$M_{\Lambda^n_k}\overline{W}G$ exists in $\mathrm{SmthMfd}$ 
and the canonical map $\overline{W}G_n\to 
M_{\Lambda^n_k}\overline{W}G$ is a surjective submersion.

Suppose we have shown that for any simplicial 
Lie group $G$, the canonical map $\overline{W}G_{n-1}\to M_{\Lambda^{n-1}_k}\overline{W}G$ 
is a surjective submersion for all $0\leq k\leq n-1$.  Let $0\leq k<n$.  
We claim that, under this assumption, the following statements are true: 
\begin{description} 
\item[$\mathrm{(a)}$] the limit $M_{\Lambda^{n-1}_{k-1}}WG$ exists in $\mathrm{SmthMfd}$, 

\item[$\mathrm{(b)}$] the map 
\[
WG_{n-1}\to M_{\Lambda^{n-1}_{k-1}}WG\times_{M_{\Lambda^{n-1}_{k-1}}\overline{W}G}\overline{W}G_{n-1} 
\]
is a surjective submersion
\end{description} 
Granted these statements, we shall show that the map in (b) is 
the canonical map 
\[
\overline{W}G_n\to M_{\Lambda^n_k}\overline{W}G.
\] 
As in \cite{Stevenson2} and the proof 
of Proposition~\ref{Lie oo-grps fib} above observe that we have 
an identification 
\[
\Lambda^n_k =   C\Lambda^{n-1}_{k}\cup_{\Lambda^{n-1}_{k}}\Delta^{n-1}.  
\]
It follows that we have an identification of sheaves on $\mathrm{CartSp}$  
\[
M_{\Lambda^n_k} \overline{W}G = M_{\Lambda^{n-1}_{k}}WG\times_{M_{\Lambda^{n-1}_{k}\overline{W}G}} 
\overline{W}G_{n-1}  
\]
which belongs to the image of $\mathrm{SmthMfd}\hookrightarrow \mathrm{Sh(CartSp)}$.  
It follows that the limit $M_{\Lambda^n_k}\overline{W}G$ exists in $\mathrm{SmthMfd}$, 
as required.  

To complete the inductive step we need to deal with the case when $k=n$.  Just as in the proof 
of Proposition~\ref{Lie oo-grps fib} above, we can settle this case by replacing the group $G$ with its opposite 
simplicial group $G^o$.

 It remains to prove the statements (a) and (b) above and the second statement 
 of the Proposition.  Before we do so, let us note 
 that in analogy with Definition~\ref{strictly principal bundle} 
we have a notion of a strictly principal bundle in simplicial manifolds, the 
only difference being that we require the bundle projection to be a 
submersion.  
\begin{definition} 
Let $G$ be a simplicial Lie group and let $X$ be a simplicial manifold.  A {\em strictly principal 
$G$-bundle} on $X$ is a simplicial manifold $P$ together with a submersion $P\to X$ 
and an action of $G$ on $P$ such that for every $n\geq 0$, the action of $G_n$ on 
$P_n$ equips $P_n\to X_n$ with the structure of a (smooth) principal $G_n$ bundle.   
\end{definition} 
   
To prove the second statement of the Proposition, and the 
statements (a) and (b) above, it is enough to prove the following lemmas.  
 
\begin{lemma} 
\label{prop: limits for horns in smth mfd}
Suppose that $P$ is a strictly principal $G$-bundle on $X$ in $\mathrm{SmthMfd}$ such that 
$P_n\to X_n$ admits a section for all $n\geq 0$.  If for some $0\leq k\leq n$ 
and some $n\geq 1$, $X_n\to M_{\Lambda^n_k}X$ is a surjective submersion, 
and the limit $M_{\Lambda^n_k}P$ exists in $\mathrm{SmthMfd}$, then 
the map 
\[
P_n\to M_{\Lambda^n_k}P\times_{M_{\Lambda^n_k}X}X_n 
\] 
is a surjective submersion and hence 
\[
P_n\to   M_{\Lambda^n_k}P
\]
is also a surjective submersion.
\end{lemma} 

\begin{lemma} 
\label{lemm: matching objects for P}
Suppose that $P$ is a strictly principal $G$-bundle on $X$ in $\mathrm{SmthMfd}$.  
Suppose that  
for some $0\leq k\leq n$ and some $n\geq 1$ the canonical map 
$X_n\to M_{\Lambda^n_k}X$ is a surjective submersion.  
Then 
\[
M_{\Lambda^n_k}P\to M_{\Lambda^n_k}X
\]
 is a 
smooth principal bundle with structure group the Lie group $M_{\Lambda^n_k}G$. 
\end{lemma}

\begin{proof}[Proof of Lemma~\ref{prop: limits for horns in smth mfd}]
Let $Y$ be an object of $\mathrm{SmthMfd}$.  Then we can form simplicial sets 
$\mathrm{SmthMfd}(Y,P)$ and $\mathrm{SmthMfd}(Y,X)$ whose sets of $n$-simplices are 
given by $\mathrm{SmthMfd}(Y,P_n)$ and $\mathrm{SmthMfd}(Y,X_n)$ respectively.  Since the functor 
$\mathrm{SmthMfd}(Y,-)$ preserves limits and the projections 
$P_n\to X_n$ admit sections for all $n\geq 0$, we see that the induced map 
\[
\mathrm{SmthMfd}(Y,P)\to \mathrm{SmthMfd}(Y,X) 
\] 
is a strictly principal bundle in $s\Set$ with structure group $\mathrm{SmthMfd}(Y,G)$.  
In particular the map 
\[
\mathrm{SmthMfd}(Y,P_n)\to M_{\Lambda^n_k}\mathrm{SmthMfd}(Y,P)\times_{M_{\Lambda^n_k}\mathrm{SmthMfd}(Y,X)} 
\mathrm{SmthMfd}(Y,X_n) 
\]
is surjective.  Taking $Y = M_{\Lambda^n_k}P\times_{M_{\Lambda^n_k}X} X_n$ we see that 
the map 
\begin{equation} 
\label{eq: smooth horn filling map}
P_n\to M_{\Lambda^n_k}P\times_{M_{\Lambda^n_k}X} X_n 
\end{equation}
admits a section.  The map~\eqref{eq: smooth horn filling map} is a morphism of principal 
bundles over $X_n$, covering the homomorphism of Lie groups 
$G_n\to M_{\Lambda^n_k}G$.  Since the smooth map underlying 
this homomorphism admits a section it follows that we 
can find a section of~\eqref{eq: smooth horn filling map} through every point of $P_n$.  
Therefore~\eqref{eq: smooth horn filling map} is a surjective submersion.  
\end{proof}


\begin{proof}[Proof of Lemma~\ref{lemm: matching objects for P}] 
The limit $M_{\Lambda^n_k}P$, if it exists, is uniquely determined 
by the requirement that $M_{\Lambda^n_k}P\to M_{\Lambda^n_k}X$ 
is a smooth $M_{\Lambda^n_k}G$ bundle, 
and that $P_n\to M_{\Lambda^n_k}P$ 
is equivariant for the homomorphism 
$h^n_k\colon G_n\to M_{\Lambda^n_k}G$.  
Since the quotient $P_n/\ker(h^n_k)$ of 
$P_n$ by the free action of the normal Lie subgroup 
$\ker(h^n_k)$ has both of these properties, it follows 
that $M_{\Lambda^n_k}P$ exists and is isomorphic to 
$P_n/\ker(h^n_k)$.  
\end{proof} 
\end{proof}

\begin{proposition}
  Let $G$ be a simplicial Lie group which presents a smooth $\infty$-group 
  in $\mathrm{Grp}(\mathrm{Smooth}\mathrm{Grpd}_{\infty})$.  Suppose that 
  $\Wbar G$ is $\mathrm{CartSp}$-acyclic (Definition~\ref{CAcyclic}). 
  Then every $G$-principal $\infty$-bundle
  over a smooth manifold 
  $X \in \mathrm{SmthMfd} \hookrightarrow \mathrm{Smooth}\mathrm{Grpd}_{\infty}$
  has a presentation by a weakly principal $G$-bundle $P\to X$ for 
  which $P$ is a locally fibrant simplicial smooth manifold and $P\to X$ is a submersive local fibration.
\end{proposition}
\proof
  By assumption of $\mathrm{CartSp}$-acyclicity 
  and theorem \ref{CAcyclicityAndLocalFibrancy}, we have that 
  $$
    \Wbar G \in [\mathrm{CartSp}^{\mathrm{op}}, \mathrm{sSet}]_{\mathrm{proj}, \mathrm{loc}}
  $$
  is fibrant. It follows that any cocycle that classifies a given
  $G$-principal $\infty$-bundle according to 
  Theorem \ref{Classification theorem for weakly G-principal bundles}
  is presented by a morphism of simplicial presheaves into $\Wbar G$ out of 
  a cofibrant resolution of $X$. By 
  Observation~\ref{ResolutionOfManifoldsByGoodCovers} we may choose this to be
  given by the {\v C}ech nerve $\check{C}(\{U_i\})\to X$ of a (differentiably) good open cover 
  $\{U_i \to X\}$ of $X$. By example \ref{CechNerveProjectionOfOpenCoverIsSubmersiveLocalFibration} this is itself a submersive local acyclic fibration.
  By Proposition~\ref{Wbar G submersively locally fibrant} 
  and Proposition~\ref{SubmersiveLocalFibrationsAreStableUnderPulback} 
  the morphism 
  $g^*WG\to \check{C}(U_i)$ in the pullback diagram of simplicial sheaves 
  \[
    \xymatrix{
	  g^* W G \ar[r] \ar[d]  & W G  \ar[d]
      \\
      \check{C}(U_i) \ar[r]^g & \Wbar G	  
	}
  \]
  is a submersive local fibration between
  locally Kan simplicial smooth manifolds.  Hence 
  so is the composite 
  $P := g^* W G \to \check{C}(U_i)\to  X$, which, by 
  Theorem~\ref{Classification theorem for weakly G-principal bundles} 
  is the principal $\infty$-bundle
  $P \xrightarrow{p} X$
  classified by $g$.  
\endofproof

\medskip

\noindent{\bf Acknowledgements.}
The writeup of this article and the companions \cite{NSSa, NSSc} was
initiated during a visit by the first two authors to the third author's institution, 
University of Glasgow, in summer 2011.  It was completed in summer 2012
when all three authors were guests at the 
Erwin Schr\"{o}dinger Institute in Vienna.  
The authors gratefully acknowledge the support of 
the Engineering and Physical Sciences Research Council 
grant number EP/I010610/1 and the support of the ESI.

\newpage

\addcontentsline{toc}{section}{References}

\end{document}